\documentclass[reqno,10pt,a4paper,oneside]{amsart}

\usepackage{amssymb,amsmath,amsthm}
\usepackage{mathtools}
\usepackage{stmaryrd}
\usepackage{enumerate}
\usepackage[all,cmtip,2cell]{xy} % cmtip gives arrow heads like in Shulman's papers
\usepackage{fouridx}
\usepackage{bbm}

\usepackage[utf8]{inputenc}
\usepackage[usenames,dvipsnames]{xcolor}
\usepackage{geometry}
\usepackage{verbatim}
\usepackage[draft=false]{hyperref}
\usepackage{cleveref}

\makeatletter
\renewenvironment{proof}[1][\proofname] {\par\pushQED{\qed}\normalfont\topsep6\p@\@plus6\p@\relax\trivlist\item[\hskip\labelsep\sc#1\@addpunct{.}]\ignorespaces}{\popQED\endtrivlist\@endpefalse}
\makeatother

\makeatletter

\renewcommand{\section}{\@startsection
{section}% name
{1}% level
{0mm}% indent
{-\baselineskip}% beforeskip
{0.5\baselineskip}% afterskip
{\large \bfseries}}% style

\renewcommand{\subsection}{\@startsection
{subsection}% name
{2}% level
{0mm}% indent
{1mm}% beforeskip
{-1ex}% afterskip
{\normalfont\normalsize\bfseries}}% styles

\makeatother

\def\defthm#1#2#3#4{
  \newtheorem{#1}[theorem]{#3}
  \newtheorem*{#1*}{#3}
  \newtheorem{#2}[theorem]{#4}
  \newtheorem*{#2*}{#4}
  \crefname{#1}{#3}{#4}
  \crefname{#2}{#4}{#4}  
}

\numberwithin{equation}{section}

\newtheoremstyle{mythm}% 
{10pt}% Space above 
{}% Space below 
{\itshape}% Body font 
{}% Indent amount 
{\sc}%  Theorem head font 
{.}% Punctuation after theorem head 
{.5em}% Space after theorem head 
{}% 

\newtheoremstyle{mydef}% 
{10pt}% Space above 
{3pt}% Space below 
{}% Body font 
{}% Indent amount 
{\sc}%  Theorem head font 
{.}% Punctuation after theorem head 
{.5em}% Space after theorem head 
{}% 

\newtheoremstyle{myrmk}% 
{10pt}% Space above 
{3pt}% Space below 
{}% Body font 
{}% Indent amount 
{\sc}%  Theorem head font 
{.}% Punctuation after theorem head 
{.5em}% Space after theorem head 
{}% 

\theoremstyle{mythm}

\newtheorem{theorem}{Theorem}[section]
\newtheorem*{theorem*}{Theorem}
\crefname{theorem}{Theorem}{Theorems}

\defthm{corollary}{corollaries}{Corollary}{Corollaries}
\defthm{lemma}{lemmata}{Lemma}{Lemmata}

\newtheorem*{replemmax}{\reptitle}
 {\end{replemmax}}

\defthm{proposition}{propositions}{Proposition}{Propositions}

\theoremstyle{mydef}
\defthm{definition}{definitions}{Definition}{Definitions}

\theoremstyle{myrmk}
\defthm{remark}{remarks}{Remark}{Remarks}
\defthm{example}{examples}{Example}{Examples}
\defthm{question}{questions}{Question}{Questions}

\crefname{section}{Section}{Sections}

\swapnumbers

\def\lam#1{{\lambda}\@lamarg#1:\@endlamarg\@ifnextchar\bgroup{.\,\lam}{.\,}}
\def\@lamarg#1:#2\@endlamarg{\if\relax\detokenize{#2}\relax #1\else\@lamvar{\@lameatcolon#2},#1\@endlamvar\fi}
\def\@lamvar#1,#2\@endlamvar{(#2\,{:}\,#1)}
\def\@lameatcolon#1:{#1}

\def\lamu#1{{\lambda}\@lamuarg#1:\@endlamuarg\@ifnextchar\bgroup{.\,\lamu}{.\,}}
\def\@lamuarg#1:#2\@endlamuarg{#1}

\newdir{ >}{{}*!/-7pt/@{>}}
\newdir{m}{->}
\newcommand{\pullback}[1]{\save*!/#1-1.2pc/#1:(-1,1)@^{|-}\restore}

\newcommand{\drpullback}{\pullback{dr}}

\newcommand{\ie}{\text{i.e.\ }}

\newcommand{\cf}{\text{cf.\,}}

\newcommand{\defeq}{=_{\operatorname{def}}}
\newcommand{\co}{\colon}
\newcommand{\iso}{\cong} 

\newcommand{\op}{\operatorname{op}}

\newcommand{\cat}[1]{\mathbb{#1}}

\newcommand{\catC}{\cat{C}}

\newcommand{\SSet}{\mathbf{SSet}}
\newcommand{\CSet}{\mathbf{CSet}}

 \UseAllTwocells

\DeclarePairedDelimiter\parens\lparen\rparen

\DeclarePairedDelimiter\braces\lbrace\rbrace

\newcommand{\arghole}{-}
\newcommand{\brarghole}{(\arghole)}

\newcommand{\cc}{\mathbin{\circ}}
\newcommand{\id}{\operatorname{id}}

\newcommand{\cod}{\operatorname{cod}}

\newcommand{\colim}{\operatorname{colim}}
\newcommand{\Id}{\operatorname{Id}}
\newcommand{\Lan}{\operatorname{Lan}}
\newcommand{\Ran}{\operatorname{Ran}}

\newcommand{\CAT}{\mathbf{CAT}}

\newcommand{\cart}{\text{\normalfont{cart}}}

\newcommand{\Adj}{\operatorname{Adj}}

\newcommand{\hatexp}{\operatorname{\widehat{exp}}}

\newcommand{\cal}[1]{\mathcal{#1}}

\newcommand{\calE}{\cal{E}}

\newcommand{\Cof}{\mathsf{Cof}}
\newcommand{\TrivFib}{\mathsf{TrivFib}}
\newcommand{\Fib}{\mathsf{Fib}}
\newcommand{\TrivCof}{\mathsf{TrivCof}}

\newcommand{\liftl}[1]{{\fourIdx{\pitchfork}{}{}{}{\smash{#1}\vphantom{I}}}}
\newcommand{\liftr}[1]{{\fourIdx{}{}{\pitchfork}{}{\smash{#1}\vphantom{I}}}}

\newcommand{\hatotimes}{\mathbin{\hat{\otimes}}}

\newcommand{\eval}{\operatorname{ev}}
\newcommand{\hateval}{\widehat{\eval}}

\newcommand{\interval}{I}

\newcommand{\thetak}{\theta^k}

\newcommand{\lcyl}{\delta^0}
\newcommand{\rcyl}{\delta^1}
\newcommand{\kcyl}{\delta^k}
\newcommand{\kcylinv}{\delta^{1-k}}
\newcommand{\ccyl}{\varepsilon}

\makeatletter
\def\ignorespacesandallpars{%
  \@ifnextchar\par
    {\expandafter\ignorespacesandallpars\@gobble}%
    {}%
}
\makeatother

\newcommand{\note}[3][]{\def\auth{#1}\textcolor{#2}{{[\ifx\auth\empty\else\auth: \fi{#3}]}}}

\newcommand{\Psh}{\mathrm{Psh}}

\newcommand{\yon}{\mathrm{y}}

\newcommand{\pAlg}[1]{{#1}\text{-}\mathsf{Map}}
\newcommand{\pCoalg}[1]{{#1}\text{-}\mathsf{Map}}
\newcommand{\palg}[1]{{#1}\text{-}\mathsf{map}}
\newcommand{\pcoalg}[1]{{#1}\text{-}\mathsf{map}}

\newcommand{\TC}{\mathsf{C_t}}
\newcommand{\FF}{\mathsf{F}}
\newcommand{\CC}{\mathsf{C}}
\newcommand{\TF}{\mathsf{F_t}}

\newcommand{\LL}{\mathsf{L}}
\newcommand{\RR}{\mathsf{R}}
\newcommand{\MM}{\mathsf{M}}

\title{The Frobenius condition, right properness, \linebreak and uniform fibrations}

\begin{document}

\begin{abstract}
We develop further the theory of weak factorization systems and algebraic weak factorization systems.
In particular, we give a method for constructing (algebraic) weak factorization systems whose right maps can be thought of as (uniform) fibrations and that satisfy the (functorial) Frobenius condition.
As applications, we obtain a new proof that the Quillen model structure for Kan complexes is right proper, avoiding entirely the use of topological realization and minimal fibrations, and we solve an open problem in the study of Voevodsky's simplicial model of type theory, proving a constructive version of the preservation of Kan fibrations by pushforward along Kan fibrations.
Our results also subsume and extend work by Coquand and others on cubical sets.
\end{abstract}

\author{Nicola Gambino}
\address{School of Mathematics, University of Leeds, Leeds LS2 9JT, UK}
\email{n.gambino@leeds.ac.uk}

\author{Christian Sattler}
\address{School of Mathematics, University of Leeds, Leeds LS2 9JT, UK}
\email{c.sattler@leeds.ac.uk}

% \date{\today}

\maketitle

\section*{Introduction}

This paper studies the Frobenius condition for weak factorization systems (wfs's), which asserts that pullback along a right map preserves left maps~\cite{garner:topological-simplicial,gambino-garner:idtypewfs}.
One reason for the interest in this condition is that it is closely related to the right properness of a Quillen model structure, which requires that pullback along a fibration preserves weak equivalences.
Indeed, for a Quillen model structure in which the cofibrations are stable under pullback, such as the Cisinski model structures~\cite{cisinski-asterisque}, right properness is equivalent to the Frobenius condition for the wfs of trivial cofibrations and fibrations.
Furthermore, when pushforward along right maps exists, \ie pullback along right maps has a right adjoint, the Frobenius condition is equivalent to the preservation of right maps by pushforward along a right map.
In the category of simplicial sets, this amounts to the preservation of Kan fibrations by pushforward along Kan fibrations, a fact that plays an important role in Voevodsky's simplicial model of univalent foundations~\cite{voevodsky-simplicial-model}.

For our development, we work in the setting of a category $\cal{E}$ equipped with an appropriate cofibrantly generated wfs~$(\Cof, \TrivFib)$, to be thought of as consisting of cofibrations and trivial fibrations, and a functorial cylinder satisfying appropriate assumptions.
Building on ideas in~\cite{cisinski-asterisque}, we then construct a new wfs $(\TrivCof, \Fib)$, to be thought of as consisting of trivial cofibrations and fibrations~(\cref{thm:wfstimes}).
Here, a fibration is defined a map with the right lifting property with respect to a class of maps called here generating trivial cofibrations, which are obtained as (a generalized form of) pushout product between the endpoint inclusions of the functorial cylinder and generating cofibrations.
We then prove that the wfs $(\TrivCof, \Fib)$ satisfies the Frobenius condition~(\cref{thm:non-alg-frobenius}).
When applied in the category of simplicial sets, with the wfs of monomorphisms and trivial Kan fibrations~\cite{quillen-homotopical} as $(\Cof, \TrivFib)$, fibrations in our sense are exactly Kan fibrations.
Therefore, our results yield a proof that the wfs of trivial cofibrations and Kan fibrations satisfies the Frobenius condition, and so give a new proof of right properness of the Quillen model structure for Kan complexes.
One reason for the interest in this proof is that it avoids entirely the use of topological realization, used in~\cite{hovey-model-categories}, and of minimal fibrations, used in~\cite{joyal-tierney:simplicial-homotopy-theory}.

We then obtain analogous results for algebraic weak factorization systems (awfs's)~\cite{garner:small-object-argument,grandis-tholen-nwfs}.
In an awfs, the left and right right maps do not just satisfy a lifting property, but are equipped with a natural family of solutions for lifting problems.
The technical ingredient underpinning this idea is to work not just with classes of maps in a fixed category $\cal{E}$, but with categories $\cal{I}$ equipped with a functor $u \co \cal{I} \to \cal{E}^\to$ (often an inclusion or a forgerful functor), where $\cal{E}^\to$ is the arrow category of $\cal{E}$.
For example, when~$\cal{E}$ is a presheaf category it is natural to define a uniform trivial cofibration to be a right map with respect to the category of monomorphisms and pullback squares, considered with the evident inclusion into $\cal{E}^\to$.
Unfolding the relevant definitions, a uniform trivial fibration is a map equipped with diagonal fillers for lifting problems against monomorphisms, with the fillers satisfying a naturality condition with respect to pullback squares between monomomorphism.
The presence of this algebraic structure remedies some unsatisfactory aspects of wfs's (such as the failure of orthogonality classes to be closed under all colimits), something that has proved to be useful for applications (see~\cite{batanin-cisinski-weber,garner:globular-operator-awfs,garner-homomorphisms} or~\cite{awodey-cubical,coquand-cubical-sets,cohen-et-al:cubicaltt,pitts-cubical-nominal,swan-awfs} for examples).
The Frobenius condition has a natural counterpart in this context, called the functorial Frobenius condition~\cite{garner:topological-simplicial}, which requires that the pullback along right maps not only preserves left maps but also preserves morphisms between them.

Starting from a suitable awfs $(\CC, \TF)$, whose left and right maps are to be thought of as uniform cofibrations and uniform trivial fibrations, we define uniform fibrations by algebraic orthogonality with respect to a category having generating trivial cofibrations as objects.
We then show that there exists an awfs $(\TC, \FF)$ that has uniform fibrations as right maps (\cref{thm:sset-cset-nwfs}) and that satisfies the functorial Frobenius condition (\cref{uniform-fibrations-uniform-frobenius}).
Therefore, pushforward along a uniform fibration preserves uniform fibrations (\cref{uniform-fibrations-frobenius-pushforward}).
An important technical step in this development is the characterization of when an algebraically-free awfs (in the sense of~\cite{garner:small-object-argument}) has the functorial Frobenius property (\cref{thm:frobenius-comparison}), which is more elaborate than its counterpart for ordinary wfs's.

In order to apply our algebraic results in simplicial sets and cubical sets, we show that, given a category of monomorphisms and pullback squares satisfying some mild assumptions, it is possible to construct an awfs $(\CC, \TF)$ satisfying the conditions necessary to apply our results (\cref{rem-lift-suitable}).
Instantiating our definitions we obtain the notion of a uniform Kan fibration, which subsumes the one introduced in~\cite{cohen-et-al:cubicaltt}.
We then establish the existence of an awfs that has uniform Kan fibrations as right maps and that satisfies the functorial Frobenius property, so that pushforward along a uniform Kan fibration preserves uniform Kan fibrations.

These results contribute to the ongoing investigations on whether Voevodsky's simplicial model of univalent foundations~\cite{voevodsky-simplicial-model} admits a constructive version, \ie one that can be defined without the law of excluded middle or the axiom of choice.
In order to do this, it is necessary to overcome the problem that the the preservation of Kan fibrations by pushforward along Kan fibrations~\cite[Lemma~2.3.1]{voevodsky-simplicial-model},
which is necessary for the interpretation of $\Pi$-types in the simplicial model, cannot be proved constructively~\cite[Section~6]{coquand-non-constructivity-kan}.
Here, we solve this problem by working with the notion of a uniform Kan fibration, which is classically equivalent to the usual one, but allows us to prove constructively the preservation of uniform Kan fibrations by pushforward along uniform Kan fibrations.
Indeed, the proof of our results in the case of simplicial sets, as well as cubical sets, are constructive (see~\cref{rem:constructive-small-object}).

Furthermore, our results shed some light on the cubical set models of type theory that have been investigated extensively in recent years~\cite{awodey-cubical,coquand-cubical-sets,cohen-et-al:cubicaltt,pitts-cubical-nominal,swan-awfs}.
Indeed, the independence results in~\cite{coquand-non-constructivity-kan} mentioned above have led the research community to shift the focus on models based on cubical sets and uniform Kan fibrations, since in that setting it is possible to develop a constructive model of univalent foundations~\cite{cohen-et-al:cubicaltt}, showing in particular that pushforward along a uniform Kan fibration preserves uniform Kan fibrations.
But, as our results show, the key aspect that makes it possible to obtain that result on pushforward is the idea of using the algebraic notion of a uniform Kan fibration, rather than working with cubical sets.
Since our results apply to cubical sets, we give also a new proof of the fact that pushforward along a uniform Kan fibration preserves uniform Kan fibrations, avoiding entirely the complex calculations with cubical sets of~\cite{cohen-et-al:cubicaltt,huber-thesis}.
Indeed, one of the initial motivations for this work was to explore whether the theory of uniform fibrations in cubical sets could be developed at a greater level of generality, so as to make it applicable also to simplicial sets.

We also provide a characterization of uniform trivial fibrations in terms of the partial map classifer (\cref{thm:part-map-classifier}), as suggested to us by Thierry Coquand and Andr\'e Joyal.
We then use this characterization to show in what sense the notion of a uniform (trivial) fibration subsumes the standard notion of a (trivial) fibration (\cref{unif-vs-non-unif}) In particular, we prove that a map in a presheaf category can be equipped with the structure of a uniform (trivial) fibration if and only if it is a (trivial) fibration.
As we observe in~\cref{why-uniform}, however, the algebraic notion of a uniform fibration is essential for our results in~\cref{sec:alg-fib} to hold constructively.

\subsection*{Remark}
In order to make the paper more concise, we confine comments on constructivity issues to certain remarks and footnotes.
However, let us point out here for the interested readers that our development does not rely on the law of the excluded middle or the axiom of choice, apart from~\cref{thm:wfstimes,thm:sset-cset-nwfs,rem-lift-suitable}.
For~\cref{thm:sset-cset-nwfs,rem-lift-suitable}, we argue that the result holds constructively for simplicial sets and cubical sets in~\cref{rem:constructive-small-object}.%

\subsection*{Organization of the paper}
\cref{sec:bac} introduces the setting in which we work throughout the paper.
The core of the paper is then organized in two parts: the first part (\cref{sec:fib,sec:frob}) studies wfs's; the second part (\cref{sec:ortf,sec:frobc,sec:alg-fib,sec:alg-frob}) studies awfs's.
In the first part, we introduce fibrations and establish the existence of the corresponding wfs in \cref{sec:fib}, and then prove the Frobenius property for this wfs in \cref{sec:frob}.
In the second part, we begin by establishing basic facts about orthogonality functors (\cref{sec:ortf}) and then establishing necessary and sufficient conditions for an algebraically-free awfs to satisfy the functorial Frobenius condition (\cref{sec:frobc}).
The rest of the second part then proceeds in parallel with the first part.
\cref{sec:alg-fib} (parallel to \cref{sec:fib}) introduces uniform fibrations and establishes the existence of the associated awfs, while~\cref{sec:alg-frob} (parallel to \cref{sec:frob}) proves the Frobenius property for uniform fibrations.
We end the paper in \cref{sec:fib-psh} studying uniform (trivial) fibrations in presheaf categories and relating (trivial) fibrations and uniform (trivial) fibrations.

\subsection*{Acknowledgements}
We are grateful to the referee for the careful reading of the paper and the
helpful suggestions.
We thank Steve Awodey, Thierry Coquand, Richard Garner, Simon Huber, Andr\'e Joyal, Andrew Pitts, Emily Riehl, and Andrew Swan for helpful discussions and comments on earlier versions of the paper, to Marco Larrea Schiavon for pointing out an erroneous statement of \cref{retract-closure} in an earlier version of this paper, and to Fosco Loregian for pointing us to a useful reference.
This material is based on research sponsored by the Air Force Research Laboratory, under agreement number FA8655-13-1-3038, by a grant from the John Templeton Foundation, and by an EPSRC grant (EP/M01729X/1).

\section{Background}
\label{sec:bac}

Throughout the paper, we work with a category $\cal{E}$, assumed to be locally presentable and locally cartesian closed.
%\footnote{We choose to work in this fixed setting for simplicity of presentation.
%However, let us note here that the only aspects of local cartesian closure used are pullbacks and the existence of right adjoints to pullback along what we later call fibrations.}
In particular, $\cal{E}$ is cocomplete as well as finitely complete and pullback preserves all limits and colimits since it has both a left adjoint and a right adjoint.
For an arrow $f \co X \to Y$, we write $f^* \co \cal{E}_{/Y} \to \cal{E}_{/X}$ for the pullback along $f$, $f_{!} \co \cal{E}_{/X} \to \cal{E}_{/Y}$ for its left adjoint, given by left composition with $f$, and $f_* \co \cal{E}_{/X} \to \cal{E}_{/Y}$ for its right adjoint, which we call pushforward along $f$.
We write $\cal{E}^\to$ for the category of arrows of $\cal{E}$ and $\cal{E}^\to_{\mathrm{cart}}$ for the subcategory of~$\cal{E}^\to$ of arrows and cartesian squares.
Examples of such categories abound; in particular, our assumptions are satisfied not just by presheaf categories, but by arbitrary Grothendieck toposes.

Our definitions and results are illustrated in two running examples.
The first is the category~$\SSet$ of simplicial sets, defined as usual to be the category of presheaves over the simplex category~$\Delta$~\cite{goerss-jardine}.
The second is the category $\CSet$ of cubical sets of~\cite{cohen-et-al:cubicaltt}.
This is defined as the category of presheaves over the category $\Box$ with objects of the form $\Box^A$ with $A$ a finite set and morphisms~$\Box^A \to \Box^B$ given by functions from $B$ to the free de Morgan algebra on
$A$.
Here, a de Morgan algebra is defined to be a bounded distributive lattice $(X, \land, \lor, \top, \bot)$ further equipped with a unary operation $\neg (-) \co X \to X$ satisfying $\neg \neg x = x$, for all $x \in X$, and de Morgan's laws: $\neg (x \land y) = \neg x \lor \neg y$, $\neg (x \lor y) = \neg x \land \neg y$, for all $x, y \in X$.

Thus,~$\CSet$ is the category of cubical sets with symmetries, diagonals, connections, and involutions.
Note that the presence of symmetries precludes $\Box$ from being a Reedy category.
We stress that the only feature of the category $\Box$ relevant for the main results are connections and symmetries.
Thus, our results apply equally to many other variations of cubical sets, excluding however those considered in~\cite{coquand-cubical-sets,huber-thesis}, which do not have connections.

Recall from~\cite{kamps-porter:homotopy} that a \emph{functorial cylinder} in $\calE$ is an endofunctor $\interval \otimes (-) \co \cal{E} \to \cal{E}$ equipped with natural transformations
\[
\lcyl \otimes (-) \co \Id_\cal{E} \to \interval \otimes (-)
\,, \quad
\rcyl \otimes (-) \co \Id_\cal{E} \to \interval \otimes (-)
\,,\]
called the left and right \emph{endpoint inclusions}, respectively.
Such a functorial cylinder is said to have \emph{contractions} if $\lcyl \otimes (-)$ and $\rcyl \otimes (-)$ have a common retraction $\ccyl \otimes (-) \co \interval \otimes (-) \to \Id_\cal{E}$, making the following diagrams commute:
\begin{equation} \label{contractions}
\begin{gathered}
\xymatrix@C+2em{
  \Id_\cal{E}
  \ar[r]^-{\lcyl \otimes (-)}
  \ar@{=}[dr]
&
  \interval \otimes (-)
  \ar[d]^(0.4){\ccyl \otimes (-)}
&
  \Id_\cal{E}
  \ar[l]_-{\rcyl \otimes (-)}
  \ar@{=}[dl]
\\&
  \Id_\cal{E}
\rlap{.}}
\end{gathered}
\end{equation}
Further, a functorial cylinder with contractions as above has \emph{connections} if, for $k \in \braces{0\,, 1}$, there is a natural transformations $c^k \otimes (-) \co \interval \otimes \interval \otimes (-) \to \interval \otimes (-)$ such that the following diagrams commute:
\begin{equation} \label{connections:0}
\begin{gathered}
\xymatrix@C+3em{
  \interval \otimes (-)
  \ar[r]^-{\kcyl \otimes \interval \otimes (-)}
  \ar[d]_{\ccyl \otimes (-)}
&
  \interval \otimes \interval \otimes (-)
  \ar[d]^{c^k \otimes (-)}
&
  \interval \otimes (-)
  \ar[l]_-{\interval \otimes \kcyl \otimes (-)}
  \ar[d]^{\ccyl \otimes (-)}
\\
  \Id_\cal{E}
  \ar[r]_{\kcyl \otimes (-)}
&
  \interval \otimes (-)
&
  \Id_\cal{E}
  \ar[l]^{\kcyl \otimes (-)}
\rlap{.}}
\end{gathered}
\end{equation}
\begin{equation} \label{connections:1}
\begin{gathered}
\xymatrix@C+3em{
  \interval \otimes (-)
  \ar[r]^-{\kcylinv \otimes \interval \otimes (-)}
  \ar@{=}[dr]
&
  \interval \otimes \interval \otimes (-)
  \ar[d]^(0.4){c^k \otimes (-)}
&
  \interval \otimes (-)
  \ar[l]_-{\interval \otimes \kcylinv \otimes (-)}
  \ar@{=}[dl]
\\&
  \interval \otimes (-)
\rlap{.}}
\end{gathered}
\end{equation}
We adopt the convention of associating the tensor product notation to the right.

If the functor $\interval\otimes (-) \co \cal{E} \to \cal{E}$ has a right adjoint,
\begin{equation} \label{equ:cyl-exp-adj}
\begin{gathered}
\xymatrix@C+3em{
  \cal{E}
  \ar@<1ex>[r]^{\interval\otimes (-)}
  \ar@{}[r]|{\bot}
&
  \cal{E}
  \ar@<1ex>[l]^{\exp(I, -)}
\rlap{,}}
\end{gathered}
\end{equation}
then we obtain a \emph{functorial cocylinder}
\[
(\exp(\interval, -), \exp(\bar{\delta}^0, -), \exp(\bar{\delta}^1, -))
\,,\]
\ie a functorial cylinder in the opposite category $\cal{E}^{\op}$~\cite{kamps-porter:homotopy}.
Contractions and connections carry over as well.
In this case, note also that $\interval\otimes (-) \co \cal{E} \to \cal{E}$ preserves colimits.
Sometimes we write $X^\interval$ for $\exp(I, X)$.

Our main examples of functorial cylinders with contractions and connections, discussed below, arise from a monoidal structure, written $(\cal{E}, \otimes, \top)$, and the presence of an \emph{interval object}, \ie an object~$\interval \in \cal{E}$ equipped with maps~$\lcyl, \rcyl \co \top \to \interval$, called the left and right \emph{endpoint inclusions}, respectively.
In such a situation, a functorial cylinder is given by tensoring with $\interval$.
It is immediate to isolate what structure on the interval object guarantees the presence of contractions and connections on the associated functorial cylinder: contractions are induced by a map $\ccyl \co \interval \to \top$ and connections by maps $c^k \co \interval \otimes \interval \to \interval$, for $k \in \braces{0\,, 1}$, provided that suitable axioms hold.
Note that if the unit $\top$ of the monoidal structure is a terminal object, then we have contractions in a unique way.

\begin{example} \label{exa:cyl-in-sset}
In $\SSet$, considered as a monoidal category with respect to its cartesian structure, an interval object is given by $\Delta^1$ with endpoint inclusions~$\kcyl \co \braces{k} \to \Delta^1$ given by special cases of horn inclusions, $\kcyl = h^1_k$, for $k \in \braces{0\,, 1}$.
We have contractions since the monoidal structure is cartesian and connections $c^k \co \Delta^1 \times \Delta^1 \to \Delta^1$, with $k \in \braces{0\,, 1}$, given on points by $c^0(x, y) = \min(x, y)$ and $c^1(x, y) = \max(x, y)$.
These are determined uniquely since~$\Delta^1$ and~$\Delta^1 \times \Delta^1$ are nerves of posets and $\min$ and $\max$ are monotone.
\end{example}

\begin{example} \label{exa:cyl-in-cuset}
The category $\Box$ has finite products given by disjoint unions of sets, which makes cubical sets into a cartesian monoidal closed category.
In $\CSet$, an interval object with contractions is given by $\Box^1$ with endpoint inclusions $\kcyl \co \braces{k} \to \Box^1$ given by the maps from a singleton set to the free de Morgan algebra on an empty set that pick the bottom element for $k = 0$ and the top element for $k =1$.
We have contractions since the unit of the monoidal structure is a terminal object.
Connection operations $c^k \co \Box^1 \otimes \Box^1 \iso \Box^2 \to \Box^1$ for $k \in \braces{0\,, 1}$ are given by the maps from a singleton set to the free de Morgan algebra on a two-element set $\braces{x\,, y}$ that pick $x \wedge y$ and $x \vee y$, respectively.
\end{example}

For the remainder of this section and the next one, we work with ordinary weak factorization systems (wfs's)~\cite{bousfield-wfs} and use standard notation and terminology.
For example, for a class $\cal{I}$ of maps in $\cal{E}$, we write $\liftr{\cal{I}}$ for the class of maps having the right lifting property with respect to all elements of $\cal{I}$.
Let us recall the definition of the Frobenius property for a wfs.

\begin{definition}
We say that a wfs $(\cal{L}, \cal{R})$ in $\cal{E}$ has the \emph{Frobenius property} if the pullback of an $\cal{L}$-map along an
$\cal{R}$-map is again an $\cal{L}$-map.
\end{definition}

See~\cite{clementino:frobenius} for an explanation of the connection between the Frobenius property for a wfs and Lawvere's Frobenius condition~\cite{lawvere-equality}.
As mentioned in the introduction, a Quillen model structure~$(\mathsf{Weq}, \Fib, \Cof)$ in which the cofibrations are stable under pullback (such as Cisinski model structures~\cite{cisinski-asterisque}, where the cofibrations are the monomorphisms) is right proper, \ie the pullback of a weak equivalence along a fibration is again a weak equivalence, if and only if the wfs~$(\TrivCof, \Fib)$ of trivial cofibrations and fibrations has the Frobenius property.

We say that a wfs $(\cal{L}, \cal{R})$ is \emph{free} on a class of maps $\cal{I}$ (not necesssarily a set) if
$\cal{R} = \liftr{\cal{I}}$.
If this is the case, then $\cal{L} = \liftl{(\liftr{\cal{I}})}$.
Of course, a wfs could be free on many different classes of maps.
By Quillen's small object argument~\cite{quillen-homotopical}, every set $\cal{I}$ determines a wfs $(\cal{L},
\cal{R})$ free on $\cal{I}$.
These are what we call \emph{cofibrantly generated wfs's}.
The next, very simple, proposition provides a convenient way of checking that a cofibrantly generated wfs satisfies the Frobenius condition and will be used in the proof of~\cref{thm:non-alg-frobenius}.

\begin{proposition} \label{thm:frobenius-equivalence}
Let $(\cal{L}, \cal{R})$ be a wfs that is free on a class of maps $\cal{I}$.
Then the following are equivalent:
\begin{enumerate}[(i)]
\item $(\cal{L}, \cal{R})$ has the Frobenius property.
\item The pushforward along an $\cal{R}$-map preserves $\cal{R}$-maps.
\item The pullback of an $\cal{I}$-map along an $\cal{R}$-map is an $\cal{L}$-map.
\end{enumerate}
\end{proposition}

\begin{proof}
The claims follow by standard properties of the interaction between lifting properties and adjoint functors.
\end{proof}

As we will see in \cref{sec:frobc}, the algebraic counterpart of \cref{thm:frobenius-equivalence}, given by \cref{thm:frobenius-comparison}, is rather more complex to prove.

\section{Fibrations}
\label{sec:fib}

From here until the end of~\cref{sec:frob} we assume that our fixed category $\cal{E}$ is equipped with a functorial cylinder $\interval\otimes (-) \co \cal{E} \to \cal{E}$ with contractions and connections as well as a right adjoint as in~\eqref{equ:cyl-exp-adj}.

The notion of a fibration that we introduce and study in this section and the next is relative not only to $\cal{E}$ and its functorial cylinder, but also to a wfs on $\cal{E}$ satisfying some assumptions that we encapsulate in the notion of a \emph{suitable wfs}, given in \cref{thm:suitable-wfs} below.
In order to state that definition, we need some notation.
For $k \in \braces{0, 1}$ and a map $i \co A \to B$ in $\cal{E}$, we define the \emph{Leibniz product}~\cite{riehl-verity:reedy} of the endpoint inclusion~$\kcyl \otimes (-)$ and~$i$ to be the map
\begin{equation} \label{equ:leib-prod}
\kcyl \hatotimes i \co (\interval \otimes A) +_A B \to \interval \otimes B
\end{equation}
determined uniquely by the following pushout diagram:
\begin{equation} \label{delta-otimes-definition}
\begin{gathered}
\xymatrix@C=1.2cm{
  A
  \ar[r]^{i}
  \ar[d]_{\kcyl \otimes A}
&
  B
  \ar@/^2pc/[ddr]^{\kcyl \otimes B}
  \ar[d]
&\\
  \interval \otimes A
  \ar@/_1pc/[drr]_{\interval \otimes i}
  \ar[r]
&
  (\interval \otimes A) +_A B
  \ar[dr]^-{\delta^k \hatotimes i}
  \pullback{ul}
&\\&&
  \interval \otimes B
\rlap{.}}
\end{gathered}
\end{equation}
We shall use Leibniz products with endpoint inclusions throughout the paper.
In the next definition and afterwards, we write $(\Cof, \TrivFib)$ to denote a suitable wfs, and refer to maps in $\Cof$ as \emph{cofibrations} and maps in $\TrivFib$ as \emph{trivial fibrations} since this helps us to convey some of the basic intuition motivating our development.

\begin{definition} \label{thm:suitable-wfs}
A wfs $(\Cof, \TrivFib)$ in $\cal{E}$ is said to be \emph{suitable} if the following hold:
\begin{enumerate}[({S}1)]
\item $(\Cof, \TrivFib)$ is cofibrantly generated.
\item Every object of $\cal{E}$ is cofibrant, \ie for every $X \in \cal{E}$, the unique map $\bot_X \co 0 \to X$ is a cofibration.
\item Cofibrations are closed under pullback, \ie for every pullback square
\[
\xymatrix{
  B
  \ar[d]_j
  \ar[r]^q
  \pullback{dr}
&
  A
  \ar[d]^i
\\
  Y
  \ar[r]_p
&
  X
\rlap{,}}
\]
if $i \in \Cof$ then $j \in \Cof$.
\item Cofibrations are closed under Leibniz product with the endpoint inclusions, \ie for every $i \in \Cof$ we have~$\kcyl \hatotimes i \in \Cof$.
\end{enumerate}
\end{definition}

\begin{example} \label{thm:generation-presheaf-cisinski}
Let $\cal{E}$ be a presheaf category.
By a result of Cisinski~\cite[Proposition~1.2.27]{cisinski-asterisque}, $\cal{E}$ admits a cofibrantly generated wfs whose cofibrations are the mononomorphisms.%
\footnote{If $\cal{E}$ is a category of presheaves over an elegant Reedy category~\cite{bergner-rezk-elegant}, it is possible to be more specific: the boundary inclusions (\ie the latching objects inclusions for representables) form a set of generating cofibrations.}
This wfs then satisfies conditions~(S1) to~(S3) of \cref{thm:suitable-wfs}.
For it to be suitable, \ie to also satisfy condition~(S4), it suffices that the endpoint inclusions $\delta^0$ and $\delta^1$ are cartesian natural transformations.
This will be the case if the functorial cylinder is induced by cartesian product with an interval object as in our main examples.
\end{example}

Let us now fix a suitable wfs $(\Cof,\TrivFib)$.
We let $\cal{I}$ be a set of maps that cofibrantly generates it, which exists by assumption (S1), and refer its elements as \emph{generating cofibrations}.
We wish to define a new wfs, to be thought of as consisting of trivial cofibrations and fibrations.
For this, we begin by defining the set of maps that determines the notion of a fibration.
This set is defined by letting
\begin{equation} \label{equ:cylinc}
\cal{I}_\otimes \defeq \braces{\kcyl \hatotimes i \mid k \in \braces{0\,, 1} \, , \; i \in \cal{I}}
\, .\end{equation}
We refer to elements of $\cal{I}_\otimes$ as \emph{generating trivial cofibrations}.

\begin{definition} \label{thm:fib}
A \emph{fibration} is a map with the right lifting property with respect to all generating trivial cofibrations.
\end{definition}

Before illustrating the notion of a fibration in our running examples, we establish that it is actually independent of the choice of the set of generating cofibrations and that fibrations are the right class of a wfs.
For the first goal, we use the cocylinder structure on the right adjoint $\exp(\interval, -)$ to $\interval \otimes (-)$.
For $k \in \braces{0 \,, 1}$ and a map $f \co X \to Y$ in $\calE$, we define the \emph{Leibniz exponential} of~$\exp(\bar{\delta}^k, -)$ and $f$ to be the map
\[
\hatexp(\bar{\delta}^k, f) \co X^\interval \to Y^\interval \times_{Y} X
\]
determined uniquely by the following pullback diagram:
\begin{equation} \label{hatexp-delta-definition}
\begin{gathered}
\xymatrix{
  X^\interval
  \ar@/^2pc/[drr]^-{\exp(\bar{\delta}^k, X)}
  \ar@/_2pc/[ddr]_{\exp(\interval, f)}
  \ar[dr]^-{\hatexp(\bar{\delta}^k, f)}
&&\\&
  Y^\interval \times_{Y} X
  \ar[d]
  \ar[r]
  \pullback{dr}
&
  X
  \ar[d]^f
\\&
  Y^\interval
  \ar[r]_{\exp(\bar{\delta}^k, Y)}
&
  Y
\rlap{.}}
\end{gathered}
\end{equation}
This definition can be extended so as to determine a functor, right adjoint to the Leibniz product functor:
\begin{equation}
\label{equ:leibniz-prod-exp}
\xymatrix@C+4em{
  \cal{E}^\to
  \ar@<1ex>[r]^{\kcyl \hatotimes (-)}
  \ar@{}[r]|{\bot}
&
  \cal{E}^\to
  \ar@<1ex>[l]^{\hatexp(\bar{\delta}^k, -)}
\rlap{.}}
\end{equation}
The desired characterization of fibrations, stated in the next proposition, now follows from standard adjointness arguments.

\begin{proposition} \label{thm:char-fib-nonalg}
For a map $p \co X \to Y$ in $\cal{E}$, the following are equivalent:
\begin{enumerate}[(i)]
\item $p$ is a fibration,
\item $\hatexp(\bar{\lcyl}, p)$ and $\hatexp(\bar{\rcyl}, p)$ are trivial fibrations. \qed
\end{enumerate}
\end{proposition}

We now define a \emph{trivial cofibration} to be a map with the left lifting property with respect to all fibrations.
We write $\Fib$ for the class of fibrations and $\TrivCof$ for the class of trivial cofibrations, so that $\Fib = \liftr{(\cal{I}_\otimes)}$ and $\TrivCof = \liftl{\Fib}$.
The next proposition is immediate.

\begin{proposition} \label{thm:wfstimes}
$(\TrivCof, \Fib)$ is a wfs in $\cal{E}$.
\end{proposition}

\begin{proof}
By Quillen's small object argument, with $\cal{I}_\otimes$ as the generating set.
\end{proof}

\begin{example} \label{thm:fib-is-kan}
In $\SSet$, let~$(\Cof, \TrivFib)$ be the wfs of monomorphisms and trivial Kan fibrations, which is cofibrantly generated by the set of boundary inclusions
\[
\cal{I} = \braces{i^n \co \partial \Delta^n \to \Delta^n \mid n \in \mathbb{N}}
.\]
Then a fibration in our sense is a map with the right lifting property with respect to all maps of the form
\[
h_k^1 \hatotimes i^n \co (\Delta^1 \times \partial \Delta^n ) \cup (\braces{k} \times \Delta^n ) \to \Delta^1 \times \Delta^n
\, ,\]
given by the Leibniz product of an endpoint inclusion $h_k^1 \co \braces{k} \to \Delta^1$ with a boundary inclusion~$i^n \co \partial \Delta^n \to \Delta^n$.
This is shown equivalent to the usual notion of Kan fibration in~\cite[Chap.~IV, Sec.~2]{gabriel-zisman:calculus-of-fractions}.
Thus, the wfs $(\TrivCof, \Fib)$ of \cref{thm:wfstimes} is the wfs of trivial cofibrations and Kan fibrations associated to the Quillen model structure for Kan complexes~\cite{quillen-homotopical}.
\end{example}

\begin{example} \label{nonalgebraic-cof}
In $\CSet$, let $(\Cof, \TrivFib)$ be the wfs in which~$\Cof$ consists of all monomorphisms, which is suitable by~\cref{thm:generation-presheaf-cisinski}.
A fibration in $\CSet$ will be called a \emph{Kan fibration}.
To illustrate the relationship with the standard Kan filling condition for cubes, let $i^n \co \partial \Box^n \to \Box^n$, for $n \in \mathbb{N}$, denote the boundary inclusion of the $n$-cube $\Box^n$, which is given by the $n$-fold Leibniz product of~$[\lcyl, \rcyl] \co \braces{0\,, 1} \to \Box^1$.
The Leibniz product in $\CSet$ can be seen to preserve monomorphisms, hence $i^n$ will be a cofibration.
Note however that the boundary inclusions $i^n$ will not form a generating set of cofibrations:
for example, the inclusion $\Box^1 \to \Box^2$ given by the diagonal does not lie in the saturation of the boundary inclusions.
The maps
\[
\lcyl \hatotimes i^n \co \sqcup_1^{1+n} \to \Box^{1+n}
\,, \quad
\rcyl \hatotimes i^n \co \sqcap_1^{1+n} \to \Box^{1+n}
\,,\]
for $n \in \mathbb{N}$, are trivial cofibrations.
Since the cube category has symmetries, a Kan fibration in $\CSet$ has the right lifting property also with respect to open box inclusions, which are counterparts in cubical sets of the horn inclusions in simplicial sets.
However, these will not form a generating set of trivial cofibrations.
\end{example}

\begin{remark} \label{cisinski-remark}
Let us also relate the notion of a fibration introduced here with the notion of a naive fibration introduced by Cisinski~\cite{cisinski-asterisque}.
So let $(\Cof, \TrivFib)$ be a wfs as in~\cref{thm:generation-presheaf-cisinski} and let $\cal{I}$ be a generating set of cofibrations.
Let $m \otimes (-) \co \Id_\cal{E} + \Id_\cal{E} \to \interval \otimes (-)$ be the natural transformation with components
\[
[\lcyl \otimes X, \rcyl \otimes X] \co X + X \to \interval \otimes X
\]
for $X \in \cal{E}$.
A \emph{naive Cisinski fibration} would be defined as a map with the right lifting property with respect to the class
\begin{align*}
\cal{I}_\otimes'
&\defeq
\braces{(m \hatotimes (-))^n (\kcyl \hatotimes i) \mid n \in \mathbb{N} \,, \; k \in \braces{0\,,1} \,, \; i \in \cal{I}}
\\
&=_{\phantom{\operatorname{def}}}
\braces{m \hatotimes \ldots \hatotimes m \hatotimes \kcyl \hatotimes i \mid k \in \braces{0\,,1} \,, \; i \in \cal{I}}
.\end{align*}
Informally, one could think of this set as the closure of the set of generating trivial cofibrations~$\cal{I}_\otimes$ under Leibniz product with $m$.
Since $\cal{I}_\otimes \subseteq \cal{I}'_\otimes$, every naive Cisinski fibration is a fibration in our sense.

Our reason for working with $\cal{I}_\otimes$ instead of this class is that they coincide in a lot of cases.
For example, assume that the functorial cylinder is induced by tensoring with an interval object in a symmetric monoidal category and~$\Cof$ is closed under Leibniz product with the boundary inclusion $[\lcyl, \rcyl] \co 1 + 1 \to \interval$, which is the case in our main examples, where $\Cof$ consists of all monomorphisms, the map $[\lcyl, \rcyl]$ is a monomorphism, and Leibniz product preserves monomorphisms.
Then, by permuting the Leibniz product, one sees that every element of $\cal{I}_\otimes'$ is isomorphic to an element of $\cal{I}_\otimes$.

To discuss a second case where fibrations and naive Cisinski fibrations coincide, we use some notions and results from~\cref{sec:frob}.
If we assume generalized connections (in the sense of $(1+n+1)$-ary operations that behave as connections for fixed middle~$n$ arguments), and that $\Cof$ is closed under Leibniz product with $[\lcyl, \rcyl]$ (for example, this is the case if $[\lcyl, \rcyl]$ is a monomorphism and the maps in~$\Cof$ are the monomorphisms in a presheaf category), then the maps in $\cal{I}'_\otimes$ can still be shown to be strong homotopy equivalences.
This gives rise to an inclusion $\cal{I}_\otimes' \subseteq \Cof \cap S$ (where $S$ is the class of strong homotopy equivalences that we define in \cref{sec:frob}).
By \cref{thm:main-sheretract}, it then follows that every fibration is a naive Cisinski fibration.
\end{remark}

We conclude this section with some observations on the wfs's $(\Cof, \TrivFib)$ and $(\TrivCof, \Fib)$.
First, since $\Cof$ is closed under Leibniz products with the endpoint inclusions, we have that~$\cal{I}_\otimes \subseteq \Cof$.
Therefore, since $\Cof$ is saturated, we have $\TrivCof \subseteq \Cof$, \ie every trivial cofibration is a cofibration.
Second, observe that \cref{thm:char-fib-nonalg} implies, by standard orthogonality arguments, that $\Fib = \liftr{(\Cof_\otimes)}$, where $\Cof_\otimes \defeq \braces{\delta^k \hatotimes i \mid k \in \braces{0\,, 1} \,,\, i \in \Cof \,}$.
Finally, we establish a corollary of~\cref{thm:char-fib-nonalg} which will be used to prove~\cref{thm:main-sheretract,thm:non-alg-frobenius-she}, which are two of the key steps in establishing the Frobenius property for the wfs $(\TrivCof, \Fib)$ in~\cref{sec:frob}.

\begin{corollary} \label{thm:kcyl-of-cof-is-trivcof-non-alg}
Let $f \co X \to Y$ in $\cal{E}$.
If $f \in \Cof$, then $\kcyl \hatotimes f \in \TrivCof$, for $k \in \braces{0 \,, 1}$.
\end{corollary}

\begin{proof}
Assume $f \in \Cof$ and let $k \in \braces{0 \,, 1}$.
In order to show that $\kcyl \hatotimes f \in \TrivCof$, it suffices to show that $\kcyl \hatotimes f \in \liftl{( \liftr{\TrivCof})}$, \ie $\kcyl \hatotimes f \in \liftl{\Fib}$.
We have
\begin{align*}
\kcyl \hatotimes f \in \liftl{\Fib}
&\Leftrightarrow
(\forall p \in \Fib) \, \kcyl \hatotimes f \pitchfork p
\\&\Leftrightarrow
(\forall p \in \Fib) \, f \pitchfork \hatexp(\bar{\delta}^k, p)
\end{align*}
But we know that if $p \in \Fib$, then $ \hatexp(\bar{\delta}^k, p) \in \TrivFib$ by~\cref{thm:char-fib-nonalg}.
\end{proof}

\section{The Frobenius property for fibrations}
\label{sec:frob}

We continue to work with our fixed locally presentable and locally cartesian closed category $\cal{E}$, equipped with a functorial cylinder $\interval\otimes (-) \co \cal{E} \to \cal{E}$ with contractions and connections, and with a right adjoint as in~\eqref{equ:cyl-exp-adj}, and a fixed suitable wfs $(\Cof, \TrivFib)$ in $\calE$, with generating set~$\cal{I}$.

Our aim in this section is to show that the wfs $(\TrivCof, \Fib)$ of \cref{thm:wfstimes} satisfies the Frobenius property.
Our proof of the Frobenius property consists of three main steps.
First, we introduce the notion of a \emph{strong homotopy equivalence} (\cref{def:strhe}) and give a characterization of strong homotopy equivalences as certain retracts (\cref{strong-h-equiv-as-section-non-alg}).
Second, we show that every generating trivial cofibration is both a cofibration and a strong homotopy equivalence and that the cofibrations which are strong homotopy equivalences are trivial cofibrations (\cref{thm:main-sheretract}).
Third, we show that the pullback of a strong homotopy equivalence along a fibration is again a strong homotopy equivalence (\cref{thm:non-alg-frobenius-she}).
With these facts in place, the Frobenius condition for $(\TrivCof, \Fib)$ follows easily (\cref{thm:non-alg-frobenius}).

We begin by introducing strong homotopy equivalences.
For this, we need to review some preliminary notions.
Let $f, g \co X \to Y$ be maps in $\cal{E}$.
Recall that a \emph{homotopy from $f$ to $g$}, denoted $\phi \co f \sim g$, is a morphism $\phi \co \interval \otimes X \to Y$ such that the following diagram commutes:
\begin{equation} \label{equ:homotopy}
\begin{gathered}
\xymatrix@C=1.2cm{
  X
  \ar[r]^-{\lcyl \otimes X}
  \ar[dr]_{f}
&
  \interval \otimes X
  \ar[d]^(0.4){\phi}
&
  X
  \ar[dl]^{g}
  \ar[l]_-{\rcyl \otimes X}
\\&
  Y
\rlap{.}}
\end{gathered}
\end{equation}
We say that a map $f \co X \to Y$ in $\cal{E}$ is called a \emph{left} (or {\emph{$0$-oriented}) \emph{homotopy equivalence} if there exist $g \co Y \to X$ and homotopies $\phi \co g \cc f \sim \id_X $, $\psi \co f \cc g \sim \id_Y$.
Dually, we have a \emph{right} (or {\emph{ $1$-oriented}) \emph{homotopy equivalence} if there exist $g \co Y \to X$ and homotopies $\phi \co \id_X \sim g \cc f$, $\psi \co \id_Y \sim f \cc g$.
The notion of a strong homotopy equivalence, defined below, is obtained by requiring an additional condition on the homotopies $\phi$ and $\psi$.

\begin{definition} \label{def:strhe}
For $k \in \braces{0\,, 1}$, a map $f \co X \to Y$ in $\cal{E}$ is called a \emph{strong $k$-oriented homotopy equivalence} if there is a map $g$ and homotopies $\phi$ and $\psi$ making $f$ into a $k$-oriented homotopy equivalence such that the following diagram commutes:
\begin{equation} \label{strength}
\begin{gathered}
\xymatrix{
  \interval \otimes X
  \ar[r]^{\interval \otimes f}
  \ar[d]_{\phi}
&
  \interval \otimes Y
  \ar[d]^{\psi}
\\
  X
  \ar[r]_{f}
&
  Y
\rlap{,}}
\end{gathered}
\end{equation}

\end{definition}

We write $S_k$ for the class of strong $k$-oriented homotopy equivalences and let~$S \defeq S_0 \cup S_1$.
Elements of $S$ are called \emph{strong homotopy equivalences}.

\begin{remark} \label{thm:endpoint-are-she}
The components of the endpoint inclusion $\kcyl \otimes (-) \co \Id_\cal{E} \to \interval \otimes (-)$ are strong $k$-oriented homotopy equivalences.
In fact, they are strong $k$-oriented deformation retracts, \ie strong $k$-oriented homotopy equivalences for which the homotopy~$\phi$ is trivial, \ie $\phi = \varepsilon \otimes X$, where $\varepsilon$ is the contraction of the functorial cylinder.
To show this, one exploits crucially the assumption that the functorial cylinder has connections.
For $k = 0$, the retraction is given by~$\varepsilon \otimes X \co \interval \otimes X \to X$ and the homotopy $\psi \co (\lcyl \otimes X) \circ (\varepsilon \otimes X) \sim 1_{\interval \otimes X}$ is given by the connection~$c^0 \otimes X$.
The correctness of the left and right endpoints of $\psi$ follow from the left sides of~\eqref{connections:0} and~\eqref{connections:1}, respectively.
The axiom for strength follows from the right part of~\eqref{connections:0}.
The case $k = 1$ is similar.
\end{remark}

Let us give an alternative characterization of strong $k$-oriented homotopy equivalences.
For this, observe that, for $k \in \braces{0\,, 1}$, we have the diagram
\begin{equation} \label{equ:thetak}
\begin{gathered}
\xymatrix@C+2em{
  X
  \ar[r]^-{\kcylinv \otimes X}
  \ar[d]_{f}
&
 \interval \otimes X
 \ar[r]^-{\iota_0}
 &
  (\interval \otimes X) +_X Y \ar[d]^{\kcyl \hatotimes f}
\\
  Y
  \ar[rr]_-{\kcylinv \otimes Y}
&&
  \interval \otimes Y
\rlap{,}}
\end{gathered}
\end{equation}
which commutes by the naturality of $\kcylinv \otimes (-)$ since $(\kcyl \hatotimes f) \cc \iota_0 = \interval \otimes f$.
This diagram gives us (naturally in $f$) a map~$\thetak \hatotimes f \co f \to \kcyl \hatotimes f$ in~$\cal{E}^\to$, which we use in the next lemma to provide a characterization of strong homotopy equivalences as retractions.

\begin{lemma} \label{strong-h-equiv-as-section-non-alg}
A map $f \co X \to Y$ is a strong $k$-oriented homotopy equivalence if and only if the map $\thetak \hatotimes f \co f \to \kcyl \hatotimes f$ exhibits $f$ as a retract of $\kcyl \hatotimes f$, \ie there are dotted arrows as follows:
\[
\xymatrix@C+4em{
  X
  \ar[r]^-{\iota_0 \cc (\kcylinv \otimes X)}
  \ar[d]_f
&
  (\interval \otimes X) +_X Y
  \ar[d]^{\kcyl \hatotimes f}
  \ar@{.>}[r]
&
  X
  \ar[d]^f
\\
  Y
  \ar[r]_-{\kcylinv \otimes Y}
&
  \interval \otimes Y
  \ar@{.>}[r]
&
  Y
\rlap{,}}
\]
such that the two horizontal composites are identities.
\end{lemma}

\begin{proof}
First, by a standard diagram-chasing argument, giving the square on the right is equivalent to giving maps $\phi \co \interval \otimes X \to X$, $g \co Y \to X$, $\psi \co \interval \otimes Y \to Y$ such that the following diagrams commute:
\begin{align} \label{equ:first-three}
\begin{gathered}
\xymatrix{
  X
  \ar[r]^-{\kcyl \otimes X}
  \ar[d]_f
&
  \interval \otimes X
  \ar[d]^{\phi}
\\
  Y \ar[r]_-{g}
&
  X
\rlap{,}}
\end{gathered}
&&
\begin{gathered}
\xymatrix{
  Y
  \ar[r]^-g
  \ar[d]_{\kcyl \otimes Y}
&
  X
  \ar[d]^f
\\
  \interval \otimes Y
  \ar[r]_-{\psi}
&
  Y
\rlap{,}}
\end{gathered}
&&
\begin{gathered}
\xymatrix{
  \interval \otimes X
  \ar[r]^-\phi
  \ar[d]_{I \otimes f}
&
  X
  \ar[d]^{f}
\\
  \interval \otimes Y
  \ar[r]_-{\psi}
&
  Y
\rlap{.}}
\end{gathered}
\end{align}
Second, requiring that the two horizontal composites are identities means that the diagrams
\begin{align} \label{equ:second-two}
\begin{gathered}
\xymatrix@C+2em{
  X
  \ar[r]^-{\kcylinv \otimes X}
  \ar@{=}[dr]
&
  \interval \otimes X
  \ar[d]^\phi
\\&
  X
\rlap{,}}
\end{gathered}
&&
\begin{gathered}
\xymatrix@C+2em{
  Y
  \ar[r]^-{\kcylinv \otimes Y}
  \ar@{=}[dr]
&
  \interval \otimes Y
  \ar[d]^{\psi}
\\&
  Y
}
\end{gathered}
\end{align}
commute.
With reference to the diagrams~\eqref{equ:homotopy} and~\eqref{strength}, the equations in~\eqref{equ:first-three} provide endpoint~$k$ for~$\phi$ and~$\psi$ as well as strength, respectively, while the equations in~\eqref{equ:second-two} provide endpoints~$1-k$ for~$\phi$ and~$\psi$, respectively.
\end{proof}

Note that strong $k$-oriented homotopy equivalences have better closure properties than strong $k$-oriented deformation retracts.
For example, \cref{strong-h-equiv-as-section-non-alg} implies that strong $k$-oriented homotopy equivalences are closed under retracts.
Indeed, let $f$ be a strong homotopy equivalence and~$g$ be a retract of~$f$.
Then $\thetak \hatotimes g$ is a retract of $\thetak \hatotimes f$ since $\thetak \hatotimes (-) \co \cal{E}^\to \to (\cal{E}^\to)^\to$ is a functor and functors preserve retracts.
Since~$f$ is a strong homotopy equivalence, $\thetak \hatotimes f$ is a section.
But then $\thetak \hatotimes g$ is also a section (since sections are closed under retracts), and so $g$ is a strong $k$-oriented homotopy equivalence.

\begin{lemma} \label{thm:kcylf-is-she}
For every map $h \co X \to Y$ in $\cal{E}$, we have that $\kcyl \hatotimes h$ is a strong $k$-oriented homotopy equivalence.
\end{lemma}

\begin{proof}
This is a diagram-chasing argument, which we outline in the case $k = 0$.
Let us abbreviate $f \defeq \delta^0 \hatotimes h \co (I \otimes X) +_X Y \to I \otimes Y$.
We define $g$ as the composite
\[
\xymatrix{
  I \otimes Y
  \ar[r]^-{\epsilon \otimes Y}
&
  Y
  \ar[r]^-{\iota_1}
&
  (I \otimes X) +_X Y
\rlap{.}}
\]
It remains to define homotopies $\phi \co g \cc f \sim \id_X$ and $\psi \co f \cc g \sim \id_Y$ such that $f \cc \phi = \psi \cc (I \otimes f)$.
The homotopy $\psi \co I \otimes I \otimes Y \to I \otimes Y$ is defined as the connection $\psi \defeq c^0 \otimes Y$.
The signature of the homotopy $\phi$ can be equivalently stated as
\[
\phi \co (I \otimes I \otimes X) +_{I \otimes X} I \otimes Y \to (I \otimes X) +_X Y
\]
since $I \otimes (-)$ preserves pushouts.
We then define $\phi = (c^0 \otimes X) +_{\epsilon \otimes X} (\epsilon \otimes Y)$.
Standard reasoning, using diagrams~\eqref{contractions}, \eqref{connections:0}, \eqref{connections:1} and naturality, verifies that $\phi$ and $\psi$ have the correct endpoints and that the $0$-oriented homotopy equivalence $(f, g, \phi, \psi)$ is strong.
\end{proof}

\cref{thm:kcylf-is-she} follows also from the more general \cref{thm:kcylf-is-she-alg}, for which we provide a more conceptual proof.

For our second step, we return to consider a suitable wfs $(\Cof, \TrivFib)$ with generating set $\cal{I}$ and the induced wfs $(\TrivCof, \Fib)$ of \cref{thm:wfstimes}, with generating set $\cal{I}_\otimes$ defined in~\eqref{equ:cylinc}.
We focus on the class $\Cof \cap S$, \ie the class of cofibrations that are strong homotopy equivalences.
The next lemma relates them to the generating trivial cofibrations and to the trivial cofibrations.

\begin{lemma} \label{thm:main-sheretract}
We have
\begin{enumerate}[(i)]
\item $\cal{I}_\otimes \subseteq \Cof \cap S$,
\item $\Cof \cap S \subseteq \TrivCof$.
\end{enumerate}
\end{lemma}

\begin{proof}
For part~(i), we have that $\cal{I}_\otimes \subseteq \Cof$ since $\cal{I} \subseteq \Cof$ and $\Cof$ is closed under Leibniz product with~$\kcyl$ (condition (S4) in \cref{thm:suitable-wfs}).
We also have $\cal{I}_\otimes \subseteq S$
since $\kcyl \hatotimes f$ is a strong $k$-oriented homotopy equivalence for every map $f$ by~\cref{thm:kcylf-is-she}.
For part~(ii), let~$f \in \Cof \cap S$.
Since $f \in S$, \cref{strong-h-equiv-as-section-non-alg} implies that $f$ is a retract of~$\kcyl \hatotimes f$ for some $k \in \braces{0\,, 1}$.
Because $f \in \Cof$, we have $\kcyl \hatotimes f \in \TrivCof$ by \cref{thm:kcyl-of-cof-is-trivcof-non-alg}.
The claim then follows since $\TrivCof$ is closed under retracts.
\end{proof}

\begin{remark} \label{fib-and-she}
By \cref{thm:main-sheretract}, we have that $\Fib = \liftr{(\Cof \cap S)}$, \ie a map is a fibration if and only if it has the right lifting property with respect to the cofibrations that are strong homotopy equivalences.
\end{remark}

The next lemma is the third main step of the proof of the Frobenius property for~$(\TrivCof, \Fib)$.

\begin{lemma} \label{thm:non-alg-frobenius-she}
For $k \in \braces{0\,, 1}$, the pullback of a strong $k$-oriented homotopy equivalence along a fibration is a strong $k$-oriented homotopy equivalence.
\end{lemma}

\begin{proof}
Let $g \in S_k$, $p \in \Fib$, and consider a pullback diagram
\begin{equation} \label{non-alg-strong-h-equiv-uniform-base-change:0}
\begin{gathered}
\xymatrix{
  A
  \ar[d]_{\bar{g}}
  \ar[r]
  \pullback{dr}
&
  B \ar[d]^{g}
\\
  X
  \ar[r]_p
&
  Y
\rlap{.}}
\end{gathered}
\end{equation}
We wish to show that $\bar{g} \in S_k$.
Since $g \in S_k$, $\thetak \hatotimes g \co g \to \kcyl \hatotimes g$ has a retraction $\rho \co \kcyl \hatotimes g \to g$.
We show that $\thetak \hatotimes \bar{g} \co \bar{g} \to \kcyl \hatotimes \bar{g}$ has a retraction $\bar{\rho} \co \kcyl \hatotimes \bar{g} \to \bar{g}$.
We define $\bar{\rho}$ so that the diagram
\[
\xymatrix@C+2em{
\bar{g}
  \ar[r]^-{\thetak \hatotimes \bar{g}}
  \ar[d]_{\sigma}
&
  \kcyl \hatotimes \bar{g}
  \ar@{.>}[r]^-{\bar{\rho}}
  \ar[d]_{\kcyl \hatotimes \sigma}
&
 \bar{g}
  \ar[d]^{\sigma}
\\
  g
  \ar[r]_-{\thetak \hatotimes g}
&
  \kcyl \hatotimes g
  \ar[r]_-{\rho}
&
  g
}
\]
in $\cal{E}^\to$ commutes, where $\sigma \co \bar{g} \to g$ is the pullback in~\eqref{non-alg-strong-h-equiv-uniform-base-change:0} and the horizontal composites should be identities.
Since $\sigma$ is a cartesian arrow with respect to the codomain fibration, it suffices to solve this problem on codomains.
Again omitting horizontal composites, we need a dotted arrow in
\[
\xymatrix@C+2em{
  X
  \ar[r]^-{\kcylinv \otimes X}
  \ar[d]_{p}
&
  \interval \otimes X
  \ar@{.>}[r]^-{\cod(\bar{\rho})}
  \ar[d]^{\interval \otimes p}
&
  X
  \ar[d]^{p}
\\
  Y
  \ar[r]_-{\kcylinv \otimes Y}
&
  \interval \otimes Y
  \ar[r]_-{\cod(\rho)}
&
  Y
\rlap{.}}
\]
This is equivalent to finding a diagonal filler in the following square:
\[
\xymatrix@C+2em{
  X
  \ar@{=}[rr]
  \ar[d]_{\kcylinv \otimes X}
&&
  X
  \ar[d]^{p}
\\
  \interval \otimes X
  \ar[r]_-{I \otimes p}
  \ar@{.>}[urr]^(0.4){\cod(\bar{\rho})}
&
  \interval \otimes Y
  \ar[r]_-{\cod(\rho)}
&
  Y
\rlap{.}}
\]
But $\kcylinv \otimes X \iso \kcylinv \hatotimes \bot_X$, and $\bot_X \co 0 \to X$ is a cofibration by the assumption that $(\Cof, \TrivFib)$ is suitable (condition~(S2) of \cref{thm:suitable-wfs}).
By~\cref{thm:kcyl-of-cof-is-trivcof-non-alg}, we have that~$\kcylinv \hatotimes \bot_X \in \TrivCof$, and hence have a diagonal filler since~$p \in \Fib$.
\end{proof}

\begin{theorem} \label{thm:non-alg-frobenius}
The wfs $(\TrivCof, \Fib)$ has the Frobenius property.
\end{theorem}

\begin{proof}
By~\cref{thm:frobenius-equivalence}, it suffices to show that, for a pullback of the form
\begin{equation*}
\begin{gathered}
\xymatrix{
  A
  \ar[d]_{j}
  \ar[r]
  \pullback{dr}
&
  B
  \ar[d]^i
\\
  X
  \ar[r]_{p}
&
  Y
\rlap{,}}
\end{gathered}
\end{equation*}
where $i \in \cal{I}_\otimes$ and $p \in \Fib$, we have $j \in \TrivCof$.
By part (i) of~\cref{thm:main-sheretract}, $i \in \Cof \cap S$.
By the closure of $\Cof$ under pullback (condition~(S3) of \cref{thm:suitable-wfs}) and~\cref{thm:non-alg-frobenius-she}, it follows that~$j \in \Cof \cap S$.
But $\Cof \cap S \subseteq \TrivCof$ by part~(ii) of~\cref{thm:main-sheretract}.
\end{proof}

\begin{example}
When applied to the category of simplicial sets (see \cref{thm:fib-is-kan}), \cref{thm:non-alg-frobenius} gives a new proof of right properness of the Quillen model structure for Kan complexes.
This proof avoids the use of topological realization, which is used in~\cite[Theorem~13.1.13]{hirschhorn-model-localizations} to deduce the desired result from the right properness of the model structure on topological spaces in which the fibrations are the Serre fibrations.
It also avoids the use of the theory of minimal fibrations, which is used in~\cite[Theorem~1.7.1]{joyal-tierney-notes} to establish the result working purely combinatorially.
\end{example}

\section{Categories of orthogonal maps}
\label{sec:ortf}

This section starts the second part of the paper, in which we are interested in generalizations of the results obtained earlier to algebraic weak factorization systems (awf's)~\cite{garner:small-object-argument,grandis-tholen-nwfs}.
In this section and the next, we will need to review and establish some facts about the algebraic setting that are useful for our development.
Importantly, in the algebraic setting we do not consider just classes of arrows in $\cal{E}$, but rather categories~$\cal{I}$, to be thought of as indexing categories, and functors~$u \co \cal{I} \to \cal{E}^\to$.
In the following, we write $u_i \co A_i \to B_i$ for the result of applying such a functor to~$i \in \cal{I}$.
Let us recall the following definition from~\cite{garner:small-object-argument}.

\begin{definition} \label{def:right-map}
Let $u \co \cal{I} \to \cal{E}^\to$.
\begin{enumerate}[(i)]
\item
A \emph{right $\cal{I}$-map} $(f, \phi) \co X \to Y$ consists of a map $f \co X \to Y$ in $\cal{E}$ and a function~$\phi$ that assigns to each $i \in \cal{I}$ and commuting square
\[
\xymatrix@C+0.5cm{
  A_i
  \ar[r]^{s}
  \ar[d]_{u_i}
&
  X
  \ar[d]^f
\\
  B_i
  \ar[r]_{t}
&
  Y
\rlap{,}}
\]
a diagonal filler $\phi(i,s, t) \co B_i \to X$, satisfying the following naturality condition: for every diagram of the form
\[
\xymatrix{
  A_i
  \ar[r]^{a}
  \ar[d]_{u_i}
&
  A_j
  \ar[r]^{s}
  \ar[d]_{u_j}
&
  X
  \ar[d]^f
\\
  B_i
  \ar[r]_{b}
&
  B_j
  \ar[r]_{t}
&
  Y
}
\]
where the left square is the image of $\sigma \co i \to j$ in $\cal{I}$ under $u$, we have that
\[
\phi(j, s, t) \cc b = \phi(i, s \cc a, t \cc b)
\, .\]
\item
A \emph{right $\cal{I}$-map morphism} $\alpha \co (f, \phi) \to (f', \phi')$ is a square $\alpha \co f \to f'$ in~$\cal{E}$ satisfying an evident compatibility condition with respect to the choices of diagonal fillers.
\end{enumerate}
\end{definition}

For $u \co \cal{I} \to \cal{E}^\to$, we write $\liftr{\cal{I}}$ for the category of right $\cal{I}$-maps and their morphisms.
We then let $\liftr{u} \co \liftr{\cal{I}} \to \cal{E}^\to$ be the evident forgetful functor, which we call the \emph{right orthogonal} of $u$.
Analogously, we can define the category $\liftl{\cal{I}}$ of left $\cal{I}$-maps and their morphisms as well as the forgetful functor $\liftl{u} \co \liftl{\cal{I}} \to \cal{E}^\to$, which we call the \emph{left orthogonal} of $u$.
As shown in~\cite[Proposition~3.8]{garner:small-object-argument}, the orthogonality operations
extend to functors forming an adjunction
\begin{equation} \label{garner-adjunction}
\begin{gathered}
\xymatrix@C+2em{
  \CAT_{/\cal{E}^\to}
  \ar@<5pt>[r]^-{\liftl{\brarghole}}
  \ar@{}[r]|-{\bot}
&
  \CAT_{/\cal{E}^\to}^{\op}
  \ar@<5pt>[l]^-{\liftr{\brarghole}}
\rlap{.}}
\end{gathered}
\end{equation}
Even if objects of $\CAT_{/\cal{E}^\to}$ are pairs of the form $(\cal{I}, u)$ where $\cal{I}$ is a category $u \co \cal{I} \to \cal{E}^\to$ and is a functor, in the following we shall often refer to them by the domain $\cal{I}$ of $u$.
A similar convention applies to maps in $\CAT_{/\cal{E}^\to}$.
For example, we write the components of the unit and counit simply as
\[
\eta_\cal{I} \co \cal{I} \to \liftl{(\liftr{\cal{I}})}
\,, \quad
\varepsilon_\cal{I} \co \cal{I} \to \liftr{(\liftl{\cal{I}})}
\,.\]
In the following, we will see that some familiar equalities between two classes of maps are replaced by functors back and forth between two categories over $\cal{E}^\to$, for which we now introduce some notation.
Given $u \co \cal{I} \to \cal{E}^\to$ and $v \co \cal{J} \to \cal{E}^\to$, we write
\begin{equation} \label{equ:logical-equivalence}
\cal{I} \leftrightarrow \cal{J}
\end{equation}
to mean that there are functors $F \co \cal{I} \to \cal{J}$ and $G \co \cal{J} \to \cal{I}$ over $\cal{E}^\to$, not necessarily forming an isomorphism or even an equivalence.

The rest of this section is devoted to establishing some facts regarding categories of orthogonal maps, describing the interplay between orthogonality functors and retract closure, slicing, adjunctions, Leibniz adjunctions, and Kan extensions.
These facts are expected generalizations of well-known statements for classes of weakly orthogonal classes in the standard setting.
Some of them follow from results in~\cite{bourke-garner-I}; the others are probably also known to experts, but we include them since they are used in the remainder of the paper and we could not find them in the literature.
We omit most proofs, which are straightforward.

Given a functor $u \co \cal{I} \to \calE^\to$, we define its~\emph{retract closure} $\overline{u} \co \overline{\cal{I}} \to \cal{E}^\to$ as follows.
An object of $\overline{\cal{I}}$ is a tuple~$(i, e, \sigma, \rho)$ consisting of an object $i \in \cal{I}$, an object $e \in \cal{E}^\to$, and maps $\sigma \co e \rightarrow u_i$, $\rho \co u_i \rightarrow e$ in $\cal{E}^\to$ which exhibit $e$ as a retract of $u_i$ in $\cal{E}^\to$, \ie such that
\[
\xymatrix{
  e
  \ar[r]^{\sigma}
  \ar@{=}[dr]
&
  u_i
  \ar[d]^{\rho}
\\&
  e
\rlap{.}}
\]
A map $(f, \kappa) \co (i, e, \sigma, \rho) \to (i', e', \sigma', \rho')$ in $\overline{\cal{I}}$ consists of a map $f \co i \to i'$ in $\cal{I}$ and a map $\kappa \co e \rightarrow e'$ in $\cal{E}^\to$ such that the following diagram in $\cal{E}^\to$ commutes:
\[
\xymatrix{
  e
  \ar[r]^{\sigma}
  \ar[d]_{\kappa}
&
  u_i
  \ar[r]^{\rho}
  \ar[d]^{u_f}
&
  e
  \ar[d]^{\kappa}
\\
  e'
  \ar[r]_{\sigma'}
&
  u_{i'}
  \ar[r]_{\rho'}
&
  e'
\rlap{.}}
\]
The functor $\overline{u} \co \overline{\cal{I}} \to \cal{E}^\to$ is then defined on objects by letting $\overline{u}(i, e, \sigma, \rho) \defeq e$, and on maps by letting $\overline{u}(f, \kappa) \defeq \kappa$.
The operation of retract closure gives a monad in an evident way.
The next proposition uses the notation introduced in~\eqref{equ:logical-equivalence}.

\begin{proposition} \label{retract-closure}
For every $u \co \cal{I} \to \cal{E}^\to$, we have pairs of functors
\[
\liftr{\parens*{\overline{\cal{I}}}} \leftrightarrow \liftr{\cal{I}}
\,, \quad
\liftl{\parens*{\overline{\cal{I}}}} \leftrightarrow \liftl{\cal{I}}
\]
over $\cal{E}^\to$.
\qed
\end{proposition}

Given $u \co \cal{I} \to \cal{E}^\to$ and $X \in \cal{E}$, we define the \emph{slice category} $\cal{I}_{/X}$ and functor $u_{/X} \co \cal{I}_{/X} \to \cal{E}_{/X}^\to$ as follows.
The category $\cal{I}_{/X}$ has as objects pairs consisting of an object $i \in \cal{I}$ and a commutative triangle in $\cal{E}$ of the form
\[
\xymatrix@C-1em{
  A_i
  \ar[dr]
  \ar[rr]^{u_i}
&&
  B_i
  \ar[dl]
\\&
  X
\rlap{.}}
\]
The functor $u_{/X} \co \cal{I}_{/X} \to \cal{E}_{/X}^\to$ sends such a pair to $u_i \co A_i \to B_i$, viewed as a morphism in~$\cal{E}_{/X}$.
There is also a \emph{coslice} category under $X$, described dually, which we denote $u_{\backslash X} \co \cal{I}_{\backslash X} \to \cal{E}^\to$.
With these definitions in place, the commutation between slicing and orthogonality functors in the algebraic setting can be stated as follows.

\begin{proposition} \label{pitchfork-slicing}
Let $u \co \cal{I} \to \cal{E}^\to$ and $X \in \cal{E}$.
\begin{enumerate}[(i)]
\item The right orthogonality functor commutes with slicing, \ie we have $\liftr{(\cal{I}_{/X})} = (\liftr{\cal{I}})_{/X}$ as categories over $\cal{E}^\to$.
\item The left orthogonality functor commutes with coslicing, \ie we have $\liftl{(\cal{I}_{\backslash X})} = (\liftl{\cal{I}})_{\backslash X}$ as categories over $\cal{E}^\to$.
\qed
\end{enumerate}
\end{proposition}

In contrast to the right orthogonality functor, the left orthogonality functor does not commute with slicing in general.
However, it does under certain assumptions, as described in \cref{pitchfork-slicing-grothendieck} below.
%Note that \cref{pitchfork-slicing-grothendieck} has an evident dual, which we do not state explicitly.

\begin{proposition} \label{pitchfork-slicing-grothendieck}
Let $u \co \cal{I} \to \cal{E}^\to$ and assume that
\[
\xymatrix@C-1em{
  \cal{I}
  \ar[rr]^{u}
  \ar[dr]_{\cod_{\cal{E}} {}\cc{} u}
&&
  \cal{E}^\to
  \ar[dl]^{\cod_{\cal{E}}}
\\&
  \cal{E}
}
\]
is a morphism of Grothendieck fibrations.
Then the left orthogonality functor commutes with slicing on $\cal{I}$, \ie for $X \in \cal{E}$ we have $\liftl{(\cal{I}_{/X})} = (\liftl{\cal{I}})_{/X}$.
\end{proposition}

\begin{proof}
First note that the composite $\cod_{\cal{E}_{/X}} {}\mathrel{\cc}{} u_{/X}$ is also a Grothendieck fibration.
When constructing the category of left maps for $\cal{I}$ or $\cal{I}_{/X}$, an application of base change and naturality of diagonal fillers shows that it is sufficient to consider lifting problems with the bottom arrow an identity.
But this restricted left orthogonality functor evidently commutes with slicing.
\end{proof}

\begin{corollary} \label{pitchfork-slicing-monad} \leavevmode
\begin{enumerate}[(i)]
\item The monad $\liftl{(\liftr{(-)})}$ commutes with slicing.
\item The monad $\liftr{(\liftl{(-)})}$ commutes with coslicing.
\end{enumerate}
\end{corollary}

\begin{proof}
Use \cref{pitchfork-slicing} and note, for the first statement, that categories of right maps satisfy the assumptions of \cref{pitchfork-slicing-grothendieck}.
The second statement follows dually.
\end{proof}

\begin{proposition}
The retract closure commutes with slicing and coslicing, in the sense that for every $u \co \cal{I} \to \cal{E}^\to$, we have $\overline{\cal{I}_{/X}} = \overline{\cal{I}}_{/X}$ and $\overline{\cal{I}_{\backslash X}} = \overline{\cal{I}}_{\backslash X}$ as categories over $\cal{E}^\to$.
\qed
\end{proposition}

\cref{lift-of-adjunction,lift-of-mates} below are implied by~\cite[Proposition~21]{bourke-garner-I} with double categories of squares specialized to categories of arrows.
\cref{lift-of-adjunction} appears also in the form of the pullback square~(6.3) in \cite{riehl-monoidal-ams}.

\begin{proposition} \label{lift-of-adjunction}
Let $u \co \cal{I} \to \cal{E}^\to$ and $v \co \cal{J} \to \cal{F}^\to$.
For an adjunction $F \dashv G$ between $\cal{E}$ and~$\cal{F}$, the following are equivalent, naturally in $\cal{I}$ and $\cal{J}$:
\begin{enumerate}[(i)]
\item the functor $F \co \cal{E}^\to \to \cal{F}^\to$ lifts to a functor $\widetilde{F} \co \cal{I} \to \liftl{\cal{J}}$ making the following diagram commute:
\[
\xymatrix@C=1.2cm{
  \cal{I}
  \ar[r]^{\widetilde{F}}
  \ar[d]_{u}
&
  \liftl{\cal{J}}
  \ar[d]^{\liftl{v}}
\\
  \cal{E}^\to
  \ar[r]_-{F}
&
  \cal{F}^\to
\rlap{,}}
\]
\item the functor $G \co \cal{F}^\to \to \cal{E}^\to$ lifts to a functor $\widetilde{G} \co \cal{J} \to \liftr{\cal{I}}$ making the following diagram commute:
\[
\xymatrix@C=1.2cm{
  \cal{J}
  \ar[d]_{v}
  \ar[r]^{\widetilde{G}}
&
  \liftr{\cal{I}}
  \ar[d]^{\liftr{u}}
\\
  \cal{F}^\to
  \ar[r]_{G}
&
  \cal{E}^\to
\rlap{.}}
\]
\end{enumerate}
\qed
\end{proposition}

\begin{proposition} \label{lift-of-mates}
Let $u \co \cal{I} \to \cal{E}^\to$ and $v \co \cal{J} \to \cal{F}^\to$.
Let $F_1 \dashv G_1$ and $F_2 \dashv G_2$ be adjunctions between $\cal{E}$ and $\cal{F}$ satisfying the equivalent conditions of \cref{lift-of-adjunction}.
Let $m \co F_1 \to F_2$ and $n \co G_2 \to G_1$ be natural transformations forming mates.
Then the following are equivalent:
\begin{enumerate}[(i)]
\item
The natural transformation $m$ lifts as follows:
\[
\xymatrix@C=1.2cm{
  \cal{I}
  \rtwocell_{\widetilde{F_2}}^{\widetilde{F_1}}{\widetilde{m}}
  \ar[d]_{u}
&
  \liftl{\cal{J}}
  \ar[d]^{\liftl{v}}
\\
  \cal{E}^\to
  \rtwocell_{F_2}^{F_1}{m}
&
  \cal{F}^\to
\rlap{,}}
\]
\item
The natural transformation $n$ lifts as follows:
\[
\xymatrix@C=1.2cm{
  \cal{J}
  \rtwocell_{\widetilde{G_1}}^{\widetilde{G_2}}{\widetilde{n}}
  \ar[d]_{u}
&
  \liftr{\cal{I}}
  \ar[d]^{\liftl{v}}
\\
  \cal{F}^\to
  \rtwocell_{G_1}^{G_2}{n}
&
  \cal{E}^\to
\rlap{.}}
\]
\qed
\end{enumerate}
\end{proposition}

We generalize \cref{lift-of-adjunction} to Leibniz adjunctions~\cite{riehl-verity:reedy}.
Let us fix functors $F \co \cal{K} \times \cal{E} \to \cal{F}$ and $G \co \cal{K}^{\op} \times \cal{F} \to \cal{E}$ related, for $k \in \cal{K}$, by an adjunction:
\[
\xymatrix@C+1em{
  \cal{E}
  \ar@<5pt>[r]^{F(k, \arghole)}
  \ar@{}[r]|{\bot}
&
  \cal{F}
  \ar@<5pt>[l]^{G(k, \arghole)}
\rlap{.}}
\]
Let $ \widehat{F} \co \cal{K}^\to \times \cal{E}^\to \to \cal{F}^\to$ and $ \widehat{G} \co (\cal{K}^{\op})^\to \times \cal{F}^\to \to \cal{E}^\to$ denote the Leibniz constructions for~$F$ and~$G^{\op}$, using pullback instead of pushout for~$\widehat{G}$.
Here and below we assume that the categories under consideration have sufficient structure to carry out the relevant Leibniz constructions.
The following proposition is a generalization of \cite[Theorem~6.5]{riehl-monoidal-ams}.

\begin{proposition} \label{lift-of-leibniz-adjunction}
Let $u \co \cal{I} \to \cal{E}^\to$, $v \co \cal{J} \to \cal{F}^\to$ be functors.
Then the following are equivalent for $h \co X \to Y$ in $\cal{K}$, naturally in $\cal{I}$ and $\cal{J}$:
\begin{enumerate}[(i)]
\item liftings $F' \co \cal{I} \to \liftl{\cal{J}}$ of $\widehat{F}(h, \arghole) \co \cal{E}^\to \to \cal{F}^\to$ making the following diagram commute:
\[
\xymatrix@C=1.2cm{
  \cal{I}
  \ar[r]^{F'}
  \ar[d]_{u}
&
  \liftl{\cal{J}}
  \ar[d]^{\liftl{v}}
\\
  \cal{E}^\to
  \ar[r]_-{\widehat{F}(h, \arghole)}
&
  \cal{F}^\to
\rlap{,}}
\]
\item liftings $G' \co \cal{J} \to \liftr{\cal{I}}$ of $\widehat{G}(h, \arghole) \co \cal{F}^\to \to \cal{E}^\to$ making the following diagram commute:
\[
\begin{gathered}
\xymatrix@C+2em{
  \cal{J}
  \ar[d]_{v}
  \ar[r]^{G'}
&
  \liftr{\cal{I}}
  \ar[d]^{\liftr{u}}
\\
  \cal{F}^\to
  \ar[r]_-{\widehat{G}(h, \arghole)}
&
  \cal{E}^\to
\rlap{.}}
\end{gathered}
\]
\qed
\end{enumerate}
\end{proposition}

\begin{corollary} \label{lift-of-leibniz-adjunction-pullback-form}
For a functor $u \co \cal{I} \to \cal{E}^\to$ and a map $h \co X \to Y$ in $\cal{K}$, there is a pullback of categories of the form
\[
\xymatrix@C=1.5cm{
  \liftr{\cal{I}}
  \ar[r]
  \ar[d]_{\liftr{(\widehat{F}(h, \arghole) \cc u)}}
  \pullback{dr}
&
  \liftr{\cal{I}}
  \ar[d]^{\liftr{u}}
\\
  \cal{F}^\to
  \ar[r]_-{\widehat{G}(h, \arghole)}
&
  \cal{E}^\to
\rlap{.}}
\]
\end{corollary}

\begin{proof}
By the Yoneda lemma and the universal property of pullbacks, it suffices to show that diagrams of the form
\[
\xymatrix{
  \cal{J}
  \ar[dr]_v
  \ar[rr]
&&
  \liftr{\cal{I}}
  \ar[dl]^{\liftr{(\widehat{F}(h, \arghole) \cc u)}}
\\&
  \cal{F}^\to
\rlap{.}}
\]
are in natural bijective correspondence with diagrams as in item~(ii) of~\cref{lift-of-leibniz-adjunction}.
But by the adjunction~\eqref{garner-adjunction}, the former are in natural bijective correspondence with diagrams as in item~(i) of~\cref{lift-of-leibniz-adjunction}.
\end{proof}

The next results spell out a special case of~\cref{lift-of-leibniz-adjunction,lift-of-leibniz-adjunction-pullback-form} which will be used in~\cref{sec:alg-fib,sec:alg-frob}.
Fix categories $\cal{E}$ and $\cal{F}$, and let $\Adj(\cal{E}, \cal{F})$ be the category of adjunctions~$U \dashv V$ (with $U \co \cal{E} \to \cal{F}$ and $V \co \cal{F} \to \cal{E}$) and maps of adjunctions~\cite[Chapter IV,\S 7]{maclane-cawm}.
We then have functors
\[
\eval_L \co \Adj(\cal{E}, \cal{F}) \times \cal{E} \to \cal{F}
\,, \quad
\eval_R \co \Adj(\cal{E}, \cal{F})^{\op} \times \cal{F} \to \cal{E}
\]
defined by evaluation of the left and right adjoint, respectively.

\begin{proposition} \label{pitchfork-leibniz-most-general-example}
Let $u \co \cal{I} \to \cal{E}^\to$ and $v \co \cal{J} \to \cal{E}^\to$ be functors.
For a map of adjunctions $\alpha \co (U, V) \to (U',V')$, the following are equivalent:
\begin{enumerate}[(i)]
\item liftings of $\hateval_L(\alpha)$ making the following diagram commute:
\[
\begin{gathered}
\xymatrix@C=1.5cm{
  \cal{I}
  \ar[r]
  \ar[d]_{u}
&
  \liftl{\cal{J}}
  \ar[d]^{\liftl{v}}
\\
  \cal{E}^\to
  \ar[r]_-{\hateval_L(\alpha, \arghole)}
&
  \cal{F}^\to
\rlap{,}}
\end{gathered}
\]
\item liftings of $\hateval_R(\alpha)$ making the following diagram commute:
\[
\begin{gathered}
\xymatrix@C+2em{
  \cal{J}
  \ar[d]_{v}
  \ar[r]
&
  \liftr{\cal{I}}
  \ar[d]^{\liftr{u}}
\\
  \cal{F}^\to
  \ar[r]_-{\hateval_R(\alpha, \arghole)}
&
  \cal{E}^\to
\rlap{.}}
\end{gathered}
\]
\end{enumerate}
\end{proposition}

\begin{corollary} \label{cor-to-most-gen-example}
For a functor $u \co \cal{I} \to \cal{E}^\to$ and a map of adjunctions $\alpha \co (U, V) \to (U',V')$, there is a pullback of categories of the form
\[
\xymatrix@C=1.5cm{
  \liftr{\cal{I}}
  \ar[r]
  \ar[d]_{\liftr{(\hateval_L(\alpha, \arghole) \cc u)}}
  \pullback{dr}
&
  \liftr{\cal{I}}
  \ar[d]^{\liftr{u}}
\\
  \cal{F}^\to
  \ar[r]_-{\hateval_R(\alpha, \arghole)}
&
  \cal{E}^\to
\rlap{.}}
\]
\end{corollary}

\begin{remark} \label{pitchfork-leibniz-most-general-applied}
Recall the adjunction between the cylinder and cocylinder functors in~\eqref{equ:cyl-exp-adj} and write $\bar{\delta}^k \co \exp(I, -) \to \Id$ for the mate of $\delta^k \co \Id \to I \otimes (-)$, for $k \in \braces{0 \,, 1}$, so that $\alpha \defeq (\delta^k, \bar{\delta}^k)$ is a map of adjunctions from $\Id \dashv \Id$ to $I \otimes (-) \dashv \exp(I, -)$.
Inspecting~\eqref{delta-otimes-definition}, we see that the left adjoint $\delta^k \hatotimes (-)$ in~\eqref{equ:leibniz-prod-exp} is $\hateval_L(\alpha, -)$.
Similarly, inspecting~\eqref{hatexp-delta-definition}, we see that the right adjoint~$\hatexp(\bar{\delta}^k, -)$ in~\eqref{equ:leibniz-prod-exp} is $\hateval_R(\alpha, -)$.
This shows that the adjunction~\eqref{equ:leibniz-prod-exp} fits into the setting of \cref{pitchfork-leibniz-most-general-example,cor-to-most-gen-example}.
\end{remark}

% \cref{lift-of-adjunction} is the special case of \cref{lift-of-leibniz-adjunction} where $\cal{K}$ is the terminal category.

Next, we record some facts about the interaction between orthogonality functors and Kan extensions along fully faithful functors.
These follow immediately from~\cite[Lemma~24]{bourke-garner-I}.

\begin{proposition} \label{kan-extension-closure}
Let $F \co \cal{I} \to \cal{J}$ be a fully faithful functor.
\begin{enumerate}[(i)]
\item Assume that the pointwise left Kan extension of $u \co \cal{I} \to \cal{E}^\to$ along $F$ exists:
\[
\xymatrix@!C@C-1em{
  \cal{I}
  \ar[dr]_{u}
  \ar[rr]^{F}
&&
  \cal{J}
  \ar[dl]^{\Lan_F u}
\\&
  \cal{E}^\to
\rlap{.}}
\]
Then the functor $\liftr{F} \co \liftr{\cal{J}} \to \liftr{\cal{I}}$, fitting in the diagram
\[
\xymatrix@!C@C-1em{
  \liftr{\cal{I}}
  \ar[dr]_{\liftr{u}}
&&
  \liftr{\cal{J}}
  \ar[ll]_{\liftr{F}}
  \ar[dl]^{\liftr{(\Lan_F u)}}
\\&
  \cal{E}^\to
\rlap{,}}
\]
is an isomorphism.
\item Assume that the pointwise right Kan extension of $u \co \cal{I} \to \cal{E}^\to$ along $F$ exists:
\[
\xymatrix@!C@C-1em{
  \cal{I}
  \ar[dr]_{u}
  \ar[rr]^{F}
&&
  \cal{J}
  \ar[dl]^{\Ran_F u}
\\&
  \cal{E}^\to
\rlap{.}}
\]
Then the functor $\liftl{F} \co \liftl{\cal{J}} \to \liftl{\cal{I}}$, fitting in the diagram
\[
\xymatrix@!C@C-1em{
  \liftl{\cal{I}}
  \ar[dr]_{\liftl{u}}
&&
  \liftl{\cal{J}}
  \ar[ll]_{\liftl{F}}
  \ar[dl]^{\liftl{(\Ran_F u)}}
\\&
  \cal{E}^\to
\rlap{,}}
\]
is an isomorphism.
\end{enumerate}
\end{proposition}

We conclude this section by proving two results in the special case that $\cal{E}$ is a presheaf category, say $\cal{E} = \Psh(\mathbb{C})$ for $\mathbb{C}$ a small category.
We write $\yon \co \mathbb{C} \to \cal{E}$ for the Yoneda embedding.
Recall from \cref{sec:bac} that we write $\cal{E}_{\cart}^\to$ for the category of arrows and pullback squares in~$\cal{E}$.

\begin{lemma} \label{left-kan-extension-of-representables}
Let $\cal{J}$ be a full subcategory of $\cal{E}_{\cart}^\to$ closed under base change to representables, \ie closed under pullback along morphisms with domain a representable presheaf.
Let $\cal{I}$ denote its restriction to arrows into representables,
\[
\xymatrix@C-1em{
  \cal{I}
  \ar[rr]
  \ar[dr]
&&
  \cal{J}
  \ar[dl]
\\&
  \cal{E}^\to
\rlap{.}}
\]
Then the inclusion $\cal{J} \to \cal{E}^\to$ is the left Kan extension of $\cal{I} \to \cal{E}^\to$ along $\cal{I} \to \cal{J}$.
\end{lemma}

\begin{proof}
Since $\cal{E}^\to$ is cocomplete, we can verify the claim using the colimit formula for left Kan extensions.
All of the following will be functorial in an object $j \co A \to B$ of $\cal{J}$.
We consider the diagram $d \co S \to \cal{E}^\to$, where $S$ is the set of pullback squares of the form
\[
\xymatrix@C=1.2cm{
  A'
  \ar[r]
  \ar[d]_{i}
  \pullback{dr}
&
  A
  \ar[d]^{j}
\\
  \yon(x)
  \ar[r]_-{b}
&
  B
}
\]
with $i \co A' \to \yon(x)$ in $\cal{I}$, and the function $d$ maps such a square to $i \in \cal{E}^\to$.
Our goal is to show that its colimit is $j$.
Using the assumption that $\cal{J}$ is closed under base change to representables, the given diagram can be described equivalently as the diagram indexed by maps $b \co \yon(x) \to B$ and valued $b^*(j)$.
The claim can then be restated as $\colim_{b \co \yon(x) \to B} b^*(j) \iso j$, which holds since pullback commutes with colimits in presheaf categories, and $\colim_{b \co \yon(x) \to B} \yon(x) \iso B$.
\end{proof}

\begin{proposition} \label{awfs-on-arrows-into-representables}
Let~$\cal{J}$ be a full subcategory of $\cal{E}_{\cart}^\to$ closed under base change to representables.
Let $\cal{I}$ denote its restriction to arrows into representables,
\[
\xymatrix@C-1em{
  \cal{I}
  \ar[rr]
  \ar[dr]
&&
  \cal{J}
  \ar[dl]
\\&
  \cal{E}^\to
\rlap{.}}
\]
Then $\liftr{\cal{I}} = \liftr{\cal{J}}$.
\end{proposition}

\begin{proof}
The result follows by combining part~(i) of \cref{kan-extension-closure} and \cref{left-kan-extension-of-representables}.
\end{proof}

\section{The functorial Frobenius condition}
\label{sec:frobc}

Recall that a \emph{functorial factorization} on a category $\cal{E}$ consist of a pair of functors $\LL, \RR \co \cal{E}^\to \to \cal{E}^\to$ providing factorizations
\[
\xymatrix{
  X
  \ar[rr]^{f}
  \ar[dr]_{\LL f}
&&
  Y
\\&
  \MM f
  \ar[ur]_{\RR f}
}
\]
functorially in~$f \in \cal{E}^\to$.
The functors $\LL$ and $\RR$ then admit a canonical copointing and pointing, respectively, given by the natural transformations $(\Id, \RR) \co \LL \to \Id$ and $(\LL, \Id) \co \Id \to \RR$.
We write $\pcoalg{\LL}$ for the category of copointed endofunctor algebras over $\LL$ and $\palg{\RR}$ for the category of pointed endofunctor algebras over $\RR$.
Both of these categories have evident forgetful functors into $\cal{E}^\to$.

For an \emph{algebraic weak factorization system} (awfs)~\cite{garner:small-object-argument,grandis-tholen-nwfs}, we require also a comultiplication $\LL \to \LL \LL$ and a multiplication $\RR \RR \to \RR$, giving the structure of a comonad on $\LL$ and of a monad on $\RR$, respectively, such that the canonical natural transformation $\LL \RR \to \RR \LL$ is a distributive law.
We write $\pCoalg{\LL}$ for the category of comonad coalgebras over~$\LL$ and $\pAlg{\RR}$ for the category of monad algebras over~$\RR$.
These come equipped with forgetful functors into $\cal{E}^\to$ and fulfill
\[
\pcoalg{\LL} = \liftl{\pAlg{\RR}}
\,, \quad
\palg{\RR} = \liftr{\pCoalg{\LL}}
\,.\]
An awfs $(\LL, \RR)$ has an underlying ordinary wfs, in which the left (right) class consists of the maps that admit the structure making them into an element of $\pcoalg{\LL}$ (of $\palg{\RR}$, respectively).
Recall also that an awfs $(\LL, \RR)$ is said to be \emph{algebraically-free} on $u \co \cal{I} \to \cal{E}^\to$ if $\pAlg{\RR} \iso \liftr{\cal{I}}$.
By Garner's small object argument~\cite{garner:small-object-argument}, every $u \co \cal{I} \to \cal{E}^\to$ with $\cal{I}$ small and $\cal{E}$ locally-presentable determines an awfs $(\LL, \RR)$ that is algebraically-free on $\cal{I}$.
We call such an awfs \emph{cofibrantly generated}.

We now state the analogue for an awfs of the Frobenius condition for a wfs.
This condition, called the functorial Frobenius condition, was stated for cloven wfs's in~\cite{garner:topological-simplicial}, which are slightly more general than awfs's.
Let us fix some notation.
For a category $\cal{E}$, we write~$\cal{E}^\to \times_\cal{E} \cal{E}^\to$ for the pullback of $\cod \co \cal{E}^\to \to \cal{E}$ with itself.
This is the category of cospans in $\cal{E}$, \ie pairs of arrows with common codomain.
Pullback gives a functor $P \co \cal{E}^\to \times_{\cal{E}} \cal{E}^\to \to \cal{E}^\to$ sending a cospan $(g, h)$ to $h^*(g)$.

\begin{definition} \label{functorial-frobenius}
An awfs $(\LL, \RR)$ in $\cal{E}$ satisfies the \emph{functorial Frobenius condition} if pullback along $\palg{\RR}$ preserves $\pcoalg{\LL}$, in the sense that we have a lift $\tilde{P}$ of $P$, as follows:
\[
\xymatrix@C+2em{
  \pcoalg{\LL} \times_{\cal{E}} \palg{\RR}
  \ar@{.>}[r]^{\tilde{P}}
  \ar[d]
&
  \pcoalg{\LL}
  \ar[d]
\\
  \cal{E}^\to \times_{\cal{E}} \cal{E}^\to
  \ar[r]_{P}
&
  \cal{E}^\to
\rlap{.}}
\]
\end{definition}

In this section, we obtain an analogue of~\cref{thm:frobenius-equivalence} for awfs's, \ie a necessary and sufficient condition for an algebraically-free awfs to satisfy the functorial Frobenius condition (\cref{thm:frobenius-comparison}).
We will apply this characterization in the next section in order to prove a functorial Frobenius property for the awfs having uniform fibrations as right maps.

In order to prove the desired characterization, it is convenient to work with a variant of the functorial Frobenius condition we introduce in~\cref{generalized-uniform-frob} below.
For this, we introduce some notation.
Given $u \co \cal{I} \to \cal{E}^\to$, we write $\cal{I}_{/\cal{E}}$ for the category with objects consisting of tuples $(X, i, a, b)$ with $X \in \cal{E}$, $i \in \cal{I}$ (giving a map $u_i \co A_i \to B_i$), $a \co A_i \to X$, $b \co B_i \to X$ such that the following diagram commutes:
\[
\xymatrix{
  A_i
  \ar[dr]_a
  \ar[rr]^{u_i}
&&
  B_i
  \ar[dl]^{b}
\\&
  X
\rlap{.}}
\]
We write $s \co \cal{I}_{/\cal{E}} \to \cal{E}$ for the first projection.
Then, for $v \co \cal{J} \to \cal{E}^\to$, we can form a pullback
\[
\xymatrix{
  \cal{I}_{/\cal{E}} \times_{\cal{E}} \cal{J}
  \ar[r]
  \ar[d]
  \pullback{dr}
&
  \cal{J}
  \ar[d]^{\cod \cc v}
\\
  \cal{I}_{/\cal{E}}
  \ar[r]_s
&
  \cal{E}
\rlap{.}}
\]
When $u$ and $v$ are the identity on $\cal{E}^\to$, we obtain
\[
\xymatrix{
  \cal{E}_{/\cal{E}}^\to \times_{\cal{E}} \cal{E}^\to
  \ar[r]
  \ar[d]
  \pullback{dr}
&
  \cal{E}^\to
  \ar[d]^{\cod}
\\
  \cal{E}^\to_{/\cal{E}}
  \ar[r]_s
&
  \cal{E}
\rlap{.}}
\]
Pullback gives a functor $Q \co \cal{E}^\to_{/\cal{E}} \times_{\cal{E}} \cal{E}^\to \to \cal{E}_{/\cal{E}}^\to$ sending
$(f, h)$ to $h^*(f)$.

\begin{definition} \label{generalized-uniform-frob}
Let $u \co \cal{I} \to \cal{E}^\to$, $v \co \cal{J} \to \cal{E}^\to$, $w \co \cal{K} \to \cal{E}^\to$ be categories of arrows.
We say that $(\cal{I}, \cal{J}, \cal{K})$ satisfies the \emph{generalized functorial Frobenius condition} if pullback along $\cal{J}$-maps sends $\cal{I}$-maps to $\cal{K}$-maps, \ie the functor $Q$ admits a lift $\tilde{Q}$ as below:
\begin{equation} \label{generalized-uniform-frob:lift}
\begin{gathered}
\xymatrix@C+2em{
  \cal{I}_{/\cal{E}} \times_{\cal{E}} \cal{J}
  \ar@{.>}[r]^{\tilde{Q}}
  \ar[d]_{u_{/\cal{E}} \times_{\cal{E}} v}
&
  \cal{K}_{/\cal{E}}
  \ar[d]^{w_{/\cal{E}}}
\\
  \cal{E}^\to_{/\cal{E}} \times_{\cal{E}} \cal{E}^\to
  \ar[r]_{Q}
&
  \cal{E}_{/\cal{E}}^\to
\rlap{.}}
\end{gathered}
\end{equation}
\end{definition}

The next proposition shows in what sense the generalized functorial Frobenius condition is functorial in its parameters.

\begin{proposition} \label{generalized-uniform-frob-functorial}
Consider functors $u_t \co \cal{I}_t \to \cal{E}^\to$, $v_t \co \cal{J}_t \to \cal{E}^\to$, $w_t \co \cal{K}_t \to \cal{E}^\to$, for $t \in \braces{1\,, 2}$, related as follows:
\begin{align*}
\xymatrix@!C@C-1.5em{
  \cal{I}_2
  \ar[rr]^{F}
  \ar[dr]_{u_2}
&&
  \cal{I}_1
  \ar[dl]^{u_1}
\\&
  \cal{E}^\to
\rlap{,}}
&&
\xymatrix@C-1.5em{
  \cal{J}_2
  \ar[rr]^-{G}
  \ar[dr]_{v_2}
&&
  \cal{J}_1
  \ar[dl]^{v_1}
\\&
  \cal{E}^\to
\rlap{,}}
&&
\xymatrix@C-1.5em{
  \cal{K}_1
  \ar[rr]^-{H}
  \ar[dr]_{w_1}
&&
  \cal{K}_2
  \ar[dl]^{w_2}
\\&
  \cal{E}^\to
\rlap{.}}
\end{align*}
If $(\cal{I}_1, \cal{J}_1, \cal{K}_1)$ satisfies the generalized functorial Frobenius condition, then so does $(\cal{I}_2, \cal{J}_2, \cal{K}_2)$.
\qed
\end{proposition}

Next, we show how the generalized functorial Frobenius condition can be simplified under some mild assumptions.

\begin{proposition} \label{generalized-uniform-frob-nicer}
Let $u \co \cal{I} \to \cal{E}^\to$, $v \co \cal{J} \to \cal{E}^\to$, $w \co \cal{K} \to \cal{E}^\to$ be categories of arrows such that
\begin{equation} \label{uniform-frob:comprehension}
\begin{gathered}
\xymatrix@C-1em{
  \cal{J}
  \ar[rr]^{v}
  \ar[dr]_{\cod_{\cal{E}} \cc v}
&&
  \cal{E}^\to
  \ar[dl]^{\cod_{\cal{E}}}
\\&
  \cal{E}
}
\end{gathered}
\end{equation}
is a morphism of Grothendieck fibrations.
Then the following are equivalent:
\begin{enumerate}[(i)]
\item the triple $(\cal{I}, \cal{J}, \cal{K})$ satisfies the generalized functorial Frobenius condition,
\item there is a lift $\tilde{P}$ as below:
\begin{equation} \label{generalized-uniform-frob-nicer:lift}
\begin{gathered}
\xymatrix@C+2em{
  \cal{I} \times_{\cal{E}} \cal{J}
  \ar@{.>}[r]^{\tilde{P}}
  \ar[d]_{u \times_{\cal{E}} v}
&
  \cal{K}
  \ar[d]^{w}
\\
  \cal{E}^\to \times_{\cal{E}} \cal{E}^\to
  \ar[r]_{P}
&
  \cal{E}^\to
\rlap{.}}
\end{gathered}
\end{equation}
\end{enumerate}
\end{proposition}

\begin{proof}
We will show that giving lifts $\tilde{Q}$ and $\tilde{P}$ as in~\eqref{generalized-uniform-frob:lift} and \eqref{generalized-uniform-frob-nicer:lift}, respectively, is equivalent.
The situation \eqref{generalized-uniform-frob:lift} a priori represents a more general scenario: we are pulling back an arrow over an object $Y$,
\[
\xymatrix{
  A
  \ar[rr]^{u_i}
  \ar[dr]
&&
  B
  \ar[dl]
\\&
  Y
\rlap{,}}
\]
with $i \in \cal{I}$, along a map $v_j \co X \to Y$ with $j \in \cal{J}$.
In contrast, in \eqref{generalized-uniform-frob-nicer:lift} we restrict to the case that the map $B \to Y$ is the identity.
It follows that any lift $\tilde{Q}$ induces a lift $\tilde{P}$.
However, under the stated assumption that \eqref{uniform-frob:comprehension} is a morphism of Grothendieck fibrations, the lift $\tilde{Q}$ can be reconstructed from the lift $\tilde{P}$ by first pulling back $v_j$ along $B \to Y$ to a map $v_{j'} \co D \to B$ with $j' \to j$ a map in $\cal{J}$ in the above situation and then setting $\tilde{Q}(i, j) \defeq \tilde{P}(i, j')$.
%If $\cal{K}$ is a category of arrows, this choice of $\tilde{Q}$ is forced upon us (up to isomorphism) by functoriality of $Q$ and the fact that $w$ reflects isomorphisms.
\end{proof}

We can now restate the functorial Frobenius condition of~\cref{functorial-frobenius} equivalently in terms of the generalized functorial Frobenius condition of~\cref{generalized-uniform-frob}.

\begin{proposition} \label{thm:alt-frob}
Let $(\mathsf{L}, \mathsf{R})$ be an awfs on $\cal{E}$.
Then the following are equivalent:
\begin{enumerate}[(i)]
\item $(\LL, \RR)$ satisfies the functorial Frobenius condition,
\item $(\pcoalg{\LL}, \palg{\RR}, \liftl{(\liftr{\pcoalg{\LL}})})$ satisfies the generalized functorial Frobenius condition.
\end{enumerate}
\end{proposition}

\begin{proof}
Recall that $\pcoalg{\LL} = \liftl{\pAlg{\RR}}$ is a category of left maps.
We claim that both~(i) and~(ii) are equivalent to $(\pcoalg{\LL}, \palg{\RR}, \pcoalg{\LL})$ satisfying the generalized functorial Frobenius condition.
For the equivalence with~(i), use \cref{generalized-uniform-frob-nicer}.
For the equivalence with~(ii), use \cref{generalized-uniform-frob-functorial}, noting that $\pcoalg{\LL} \leftrightarrow \liftl{(\liftr{\pcoalg{\LL}})}$ over $\cal{E}^\to$, using the adjunction~\eqref{garner-adjunction}.
\end{proof}

The benefit of the reformulation in~\cref{thm:alt-frob} of the functorial Frobenius condition as an instance of the generalized functorial Frobenius condition is that the latter admits an equivalent rephrasing in terms of pushforward, rather than pullback, functors.
This rephrasing, which does not seem directly possible for the functorial Frobenius condition, will be essential to our characterization of algebraically-free awfs's satisfying the functorial Frobenius condition.
As a first step towards the pushforward formulation, we decompose the generalized functorial Frobenius condition into an object part and a morphism part.
For this, recall that any arrow $f \co X \to Y$ induces an adjunction~$f_{!} \dashv f^*$ between the slices $\cal{E}_{/X}$ and $\cal{E}_{/Y}$ given by left composition~$f_{!}$ and pullback~$f^*$.
Given a functor $u \co \cal{I} \to \cal{E}^\to$, it is immediate to check that the left composition functor lifts as follows:
\[
\xymatrix@C+1em{
  \cal{I}_{/X}
  \ar[r]^-{\widetilde{f_!}}
  \ar[d]_{u_{/X}}
&
  \cal{I}_{/Y}
  \ar[d]^{u_{/Y}}
\\
  \calE_{/X}^\to
  \ar[r]_-{f_!}
&
  \calE_{/Y}^\to
\rlap{.}}
\]
Note that we can view the bottom map as a lift of $f_! \co \cal{E}_{/X} \to \cal{E}_{/Y}$ to arrow categories.
Composing with the unit $\cal{I}_{/Y} \to \liftl{(\liftr{(\cal{I}_{/Y})})}$ of the adjunction~\eqref{garner-adjunction}, applying \cref{lift-of-adjunction}, and using \cref{pitchfork-slicing} to commute orthogonality functors with slicing, the pullback functor $f^* \co \cal{E}_{/Y} \to \cal{E}_{/X}$ then lifts to slices of the right orthogonality categories,
\[
\xymatrix@C=1.5cm{
  \liftr{\cal{I}}_{/Y}
  \ar[d]_{{\liftr{u}}_{/Y}}
  \ar[r]^{\widetilde{f^*}}
&
  \liftr{\cal{I}}_{/X}
  \ar[d]^{{\liftr{u}}_{/X}}
\\
  {\cal{E}}_{/Y}^\to
  \ar[r]_{f^*}
&
  \cal{E}_{/X}^\to
\rlap{.}}
\]

\begin{proposition} \label{def:uniFrobcond}
Let $u \co \cal{I} \to \cal{E}^\to$, $v \co \cal{J} \to \cal{E}^\to$, $w \co \cal{K} \to \cal{E}^\to$.
Then $(\cal{I}, \cal{J}, \liftl{(\liftr{\cal{K}})})$ satisfies the generalized functorial Frobenius property if and only if the following hold:
\begin{enumerate}[(i)]
\item for every $j \in \cal{J}$, pullback along $v_j \co C_j \to D_j$ lifts to a functor
\[
\xymatrix@C+3em{
  \cal{I}_{/Y}
  \ar[r]^-{\widetilde{v_j^*}}
  \ar[d]_-{u_{/D_j}}
&
  \liftl{(\liftr{\cal{K}})}_{/C_j}
  \ar[d]^-{\liftl{(\liftr{w})}_{/C_j}}
\\
  \cal{E}_{/D_j}^\to \ar[r]_-{v_j^*}
&
  \cal{E}_{/C_j}^\to
\rlap{.}}
\]
\item for every $\tau \co j \to k$ in $\cal{J}$, the square
\begin{equation} \label{def:beck-chevalley:0}
\begin{gathered}
\xymatrix@C+3em{
  C_j
  \ar[r]^{v_j}
  \ar[d]_{s}
&
  D_j
  \ar[d]^{t}
\\
  C_{k}
  \ar[r]_-{v_{k}}
&
  D_{k}
}
\end{gathered}
\end{equation}
induced by $v_\tau \co v_j \to v_k$ is such that the canonical natural transformation
\[
\xymatrix@C+3em{
  \cal{E}^\to_{/D_j}
  \ar[d]_{t_!}
  \ar[r]^{v_j^*}
  \ar@{}[dr]|{\textstyle\Downarrow \rlap{$\labelstyle\phi$}}
&
  \cal{E}^\to_{/C_j}
  \ar[d]^{s_!}
\\
  \cal{E}^\to_{/D_k}
  \ar[r]_{v_k^*}
&
  \cal{E}^\to_{/C_k}
}
\]
lifts to a natural transformation
\[
\xymatrix@C+3em{
  \cal{I}_{/D_j}
  \ar[r]^{\widetilde{v_j^*}}
  \ar[d]_{\widetilde{t_!}}
  \ar@{}[dr]|{\textstyle\Downarrow \rlap{$\labelstyle\tilde{\phi}$}}
&
  \liftl{(\liftr{\cal{K}})}_{/C_j}
  \ar[d]^{\widetilde{s_!}}
\\
  \cal{I}_{/D_k}
  \ar[r]_{\widetilde{v_k^*}}
&
  \liftl{(\liftr{\cal{K}})}_{/C_k}
\rlap{.}}
\]
\end{enumerate}
\end{proposition}

\begin{proof}
We consider the lift $\widetilde{Q}$ of $Q$ in the generalized functorial Frobenius condition~\eqref{generalized-uniform-frob:lift} for~$(\cal{I}, \cal{J}, \liftl{(\liftr{\cal{K}})})$.
Since the functor $\liftl{(\liftr{w})} \co \liftl{(\liftr{\cal{K}})} \to \cal{E}^\to$ is faithful, the lift~$\widetilde{Q}$ consists just of a lift of the action of $Q$ on objects that is coherent (in an evident sense) with respect to the action of~$Q$ on morphisms, in the sense that the action of $Q$ on morphisms then determines the action of~$\widetilde{Q}$ on morphisms.
Coherence in morphisms in $\cal{I}_{/\cal{E}} \times_{\cal{E}} \cal{J}$ of the action on objects of $\widetilde{Q}$ separates into two parts:
\begin{enumerate}[(1)]
\item coherence in morphisms in $\cal{I}_{/D_j}$ for fixed $v_j \co C_j \to D_j$ with $j \in \cal{J}$,
\item coherence in morphisms in $\cal{J}$.
\end{enumerate}
The action of $\widetilde{Q}$ on objects together with coherence~(1) constitutes part~(i).
Coherence~(2) constitutes part~(ii).
\end{proof}

Note that if the commutative square~\eqref{def:beck-chevalley:0} is a pullback, then the canonical natural transformation~$\phi$ is an isomorphism by the usual Beck-Chevalley condition, and so is~$\tilde{\phi}$ since $\liftl{(\liftr{u})}$ reflects isomorphisms.

\cref{def:uniFrobcond} has an immediate dual, expressed in terms of pushforward rather than pullback, which we state next.

\begin{proposition} \label{lift-pushforward}
Let $u \co \cal{I} \to \cal{E}^\to$, $v \co \cal{J} \to \cal{E}^\to$, $w \co \cal{K} \to \cal{E}^\to$.
Then $(\cal{I}, \cal{J}, \liftl{(\liftr{\cal{K}})})$ satisfies the generalized Frobenius property if and only if the following hold:
\begin{enumerate}[(i)]
\item for every $j \in \cal{J}$, pushforward along $v_j \co C_j \to D_j$ lifts to a functor
\[
\xymatrix@C=1.5cm{
  {\liftr{\cal{K}}}_{/C_j}
  \ar[r]^{\widetilde{(v_j)_*}}
  \ar[d]_{\liftr{w}_{/C_j}}
&
  {\liftr{\cal{I}}}_{/D_j}
  \ar[d]^{{\liftr{u}}_{/D_j}}
\\
  \cal{E}_{/C_j}^\to
  \ar[r]_{(v_j)_*}
&
  \cal{E}_{/D_j}^\to
\rlap{.}}
\]
\item for every $\tau \co j \to k$, the square~~\eqref{def:beck-chevalley:0} induced by $v_\tau \co v_j \to v_k$ is such that the canonical natural transformation
\[
\xymatrix@C+2em{
  \cal{E}^\to_{/C_k}
  \ar[r]^{(v_k)_*}
  \ar[d]_{s^*}
  \ar@{}[dr]|{\textstyle\Downarrow \rlap{$\labelstyle\psi$}}
&
  \cal{E}^\to_{/D_k}
  \ar[d]^{t^*}
\\
  \cal{E}^\to_{/C_j}
  \ar[r]_{(v_j)_*}
&
  \cal{E}^\to_{/D_j}
}
\]
lifts to a natural transformation
\[
\xymatrix@C+2em{
  {\liftr{\cal{K}}}_{/C_k}
  \ar[r]^{\widetilde{(v_k)_*}}
  \ar[d]_{\widetilde{s^*}}
  \ar@{}[dr]|{\textstyle\Downarrow \rlap{$\labelstyle\tilde{\psi}$}}
&
  {\liftr{\cal{I}}}_{/D_k}
  \ar[d]^{\widetilde{t^*}}
\\
  {\liftr{\cal{K}}}_{/C_j}
  \ar[r]_{\widetilde{(v_j)_*}}
&
  {\liftr{\cal{I}}}_{/D_j}
\rlap{.}}
\]
\end{enumerate}
\end{proposition}

\begin{proof}
For (i), note that we can view the bottom map in part (i) of~\cref{def:uniFrobcond} as a lift of~$v_j^* \co \cal{E}_{/D_j} \to \cal{E}_{/C_j}$ to arrow categories, so we can apply \cref{lift-of-adjunction} to the adjunction $v_j^* \dashv (v_j)_*$ with $u = u_{/D_j}$ and $v = {\liftr{w}}_{/C_j}$, using \cref{pitchfork-slicing-grothendieck} to permute slicing and orthogonality functors.
For (ii), apply \cref{lift-of-mates} to the adjunctions $s_! v_j^* \dashv (v_j)_* s^*$ and $v_k^* t_! \dashv t^* (v_k)_*$ recalling that $\phi$ and $\psi$ are mates.
\end{proof}

\begin{proposition} \label{double-pitchfork-in-generalized-uniform-frob}
Let $u \co \cal{I} \to \cal{E}^\to$ and $v \co \cal{J} \to \cal{E}^\to$.
Then the following are equivalent:
\begin{enumerate}[(i)]
\item $(\cal{I}, \cal{J}, \liftl{(\liftr{\cal{I}})})$ satisfies the generalized functorial Frobenius condition,
\item $(\liftl{(\liftr{\cal{I}})}, \cal{J}, \liftl{(\liftr{\cal{I}})})$ satisfies the generalized functorial Frobenius condition.
\end{enumerate}
\end{proposition}

\begin{proof}
Using \cref{generalized-uniform-frob-functorial} and the unit $\cal{I} \to \liftl{(\liftr{\cal{I}})}$ of the adjunction~\eqref{garner-adjunction}, we see that statement~(ii) implies statement~(i).
For the converse direction, we use the characterizations of statements~(i) and~(ii) provided by \cref{lift-pushforward}.
The lift of pushforward for~(i) can be composed with the counit $\liftr{\cal{I}} \to \liftr{(\liftl{(\liftr{\cal{I}})})}$ of the adjunction~\eqref{garner-adjunction} to induce a lift of pushforward for~(ii).
\end{proof}

We can finally give the desired characterization of the functorial Frobenius condition for an algebraically-free awfs.

\begin{theorem} \label{thm:frobenius-comparison}
Let $(\LL, \RR)$ be an awfs algebraically-free on $u \co \cal{I} \to \cal{E}^\to$.
Then the following are equivalent:
\begin{enumerate}[(i)]
\item $(\LL, \RR)$ satisfies the functorial Frobenius condition,
\item $(\cal{I}, \liftr{\cal{I}}, \liftl{(\liftr{\cal{I}})}) = (\cal{I}, \pAlg{\RR}, \pcoalg{\LL})$ satisfies the generalized functorial Frobenius condition.
\end{enumerate}
\end{theorem}

\begin{proof}
Using \cref{retract-closure} and the adjunction~\eqref{garner-adjunction}, we have $\overline{\liftr{\cal{I}}} \leftrightarrow \liftr{\cal{I}}$, hence $\pAlg{\RR} \leftrightarrow \overline{\pAlg{\RR}}$.
Note that $\overline{\pAlg{\RR}} \leftrightarrow \palg{\RR}$ (this is the case for pointed endofunctor and monad algebras over any monad).
It follows that $\pAlg{\RR} \leftrightarrow \palg{\RR}$.

With this, the functorial Frobenius condition for $(\LL, \RR)$ can be stated equivalently as the existence of a lift of the form
\begin{equation} \label{frobenius-comparison:0}
\begin{gathered}
\xymatrix@C+2em{
  \liftl{(\liftr{\cal{I}})} \times_{\cal{E}} \liftr{\cal{I}}
  \ar@{.>}[r]
  \ar[d]
&
  \liftl{(\liftr{\cal{I}})}
  \ar[d]
\\
  \cal{E}^\to \times_{\cal{E}} \cal{E}^\to
  \ar[r]_{P'}
&
  \cal{E}^\to
\rlap{.}}
\end{gathered}
\end{equation}
This amounts to the generalized functorial Frobenius condition for $(\liftl{(\liftr{\cal{I}})}, \liftr{\cal{I}}, \liftl{(\liftr{\cal{I}})})$ by \cref{generalized-uniform-frob-nicer}, which is equivalent to the generalized functorial Frobenius condition for $(\cal{I}, \liftr{\cal{I}}, \liftl{(\liftr{\cal{I}})})$ by \cref{double-pitchfork-in-generalized-uniform-frob}.
\end{proof}

We conclude the section with a result needed in \cref{sec:alg-frob}.

\begin{proposition} \label{generalized-uniform-frob-product-u}
Let $u_t \co \cal{I}_t \to \cal{E}^\to$, $w_t \co \cal{K}_t \to \cal{E}^\to$ for $t \in \braces{1\,, 2}$ and $v \co \cal{J} \to \cal{E}^\to$.
If $(\cal{I}_1, \cal{J}, \cal{K}_1,)$ and $(\cal{I}_2, \cal{J}, \cal{K}_2)$ satisfy the generalized functorial Frobenius property, then so do
\[
(\cal{I}_1 \times_{\cal{E}^\to} \cal{I}_2, \cal{J}, \cal{K}_1 \times_{\cal{E}^\to} \cal{K}_2)
\,, \quad
(\cal{I}_1 +_{\cal{E}^\to} \cal{I}_2, \cal{J}, \cal{K}_1 +_{\cal{E}^\to} \cal{K}_2)
\,.\]
\end{proposition}

\begin{proof}
This follows easily from \cref{generalized-uniform-frob-functorial} using the universal property of products and coproducts in the category $\CAT_{/\cal{E}^\to}$.
\end{proof}

This concludes our work on orthogonality functors and the functorial Frobenius condition in general.
In \cref{sec:alg-fib,sec:alg-frob}, we develop the algebraic counterpart of the material in \cref{sec:fib,sec:frob}.

\section{Uniform fibrations}
\label{sec:alg-fib}

This section is devoted to introducing the notion of a uniform fibration, which is the algebraic counterpart of the notion of a fibration introduced in~\cref{thm:fib}, and establishing some basic results about it.
For this, we first introduce suitable awfs's, in complete analogy with the way we defined suitable wfs's in~\cref{thm:suitable-wfs}.
We will use the notation $(\CC, \TF)$ to denote a suitable awfs and let
\begin{equation} \label{equ:Mapmap}
\Cof \defeq \pcoalg{\CC}
\,, \quad
\TrivFib \defeq \pAlg{\TF}
\,.
\end{equation}
We then refer to objects in $\Cof$ as \emph{uniform cofibrations} and objects in $\TrivFib$ as \emph{uniform trivial fibrations}.
Again, we hope that this helps to convey some of the basic intuition motivating our development.
In~\eqref{equ:Mapmap}, the dichotomy between copointed endofunctor coalgebras on one hand and monad algebras on the other is motivated by our technical development: the uniform fibrations introduced in~\cref{def:I-fibration}, which generalize the uniform fibrations in cubical sets in~\cite{cohen-et-al:cubicaltt}, are monad algebras (as shown in \cref{thm:sset-cset-nwfs}) rather than pointed endofunctor algebras.
Note however that we will only consider cofibrantly generated awfs $(\LL, \RR)$, for which there are always functors back and forth between $\pAlg{\RR}$ and $\palg{\RR}$.

\begin{definition} \label{thm:suitable-awfs}
An awfs $(\CC, \TF)$ is said to be \emph{suitable} if the following conditions hold.
\begin{enumerate}[({S}1)]
\item $(\CC, \TF)$ is cofibrantly generated.
\item The functor $\bot \co \cal{E} \to \cal{E}^\to$ mapping $X \in \cal{E}$ to $\bot_X \co 0_{\cal{E}} \to X$ factors through $\Cof$, as in the diagram
\[
\xymatrix@C+3em{
  \cal{E}
  \ar[r]^{\widetilde{\bot}}
  \ar@/_1pc/[dr]_{\bot}
&
  \Cof
  \ar[d]
\\&
  \cal{E}^\to
\rlap{.}}
\]
\item $\Cof$ is closed under pullback, in the sense that the pullback functor $P$ of \cref{sec:frobc} lifts as follows:
\[
\xymatrix@C+3em{
  \Cof \times_{\cal{E}} \cal{E}^\to
  \ar@{.>}[r]^{\widetilde{P}}
  \ar[d]
&
  \Cof
  \ar[d]
\\
  \cal{E}^\to \times_{\cal{E}} \calE^\to
  \ar[r]_{P}
&
  \cal{E}^\to
\rlap{.}}
\]
\item $\Cof$ is closed under Leibniz product with endpoint inclusions, in the sense that the Leibniz product functor lifts as follows for $k \in \braces{0\,,1}$:
\[
\xymatrix@C+3em{
  \Cof
  \ar@{.>}[r]^{\widetilde{\kcyl \hatotimes (-)}}
  \ar[d]
&
 \Cof
  \ar[d]
\\
  \cal{E}^\to
  \ar[r]_{\kcyl \hatotimes (-)}
&
  \cal{E}^\to
\rlap{.}}
\]
\end{enumerate}
\end{definition}

The main difference between the notion of a suitable wfs~(\cref{thm:suitable-wfs}) and that of a suitable awfs (\cref{thm:suitable-awfs}) is that the conditions in the former express closure of the class of left maps under the action of certain functions, while the conditions in the latter express closure of the category of left maps under the action of certain functors.

Note that, by \cref{generalized-uniform-frob-nicer}, condition~(S3) is equivalent to the generalized functorial Frobenius condition for $(\Cof, \cal{E}^\to, \Cof)$.

\begin{example} \label{unif-triv-fib-sset}
As we will see in~\cref{sec:fib-psh}, the categories $\SSet$ and $\CSet$ admit suitable awfs's $(\CC, \TF)$ that are algebraically-free on the categories $\cal{M}$ of all monomorphisms and pullback squares, \ie such that $\TrivFib = \liftr{\cal{M}}$.
The right maps of these awfs's are a natural algebraic counterpart of trivial Kan fibrations and so we call them \emph{uniform trivial Kan fibrations}.
Similarly, we will see in~\cref{justify-sset-cset-examples} that the category $\CSet$ also admits a suitable awfs $(\CC', \TF')$ algebraically-free on the category $\cal{M}'$ of monomorphisms classified by the face lattice $\Phi$ of~\cite{cohen-et-al:cubicaltt} and pullback squares (see \cref{nonalgebraic-cof}).
\end{example}

Let us now fix a suitable awfs $(\CC, \TF)$ cofibrantly generated by a small category of maps $u \co \cal{I} \to \cal{E}^\to$.
We refer to the objects of $\cal{I}$ as \emph{generating uniform cofibrations}.
For $k \in \braces{0\,, 1}$, we define a functor $\kcyl \hatotimes u \co \cal{I} \to \cal{E}^\to$ by letting
\[
(\kcyl \hatotimes u)_i \defeq \kcyl \hatotimes u_i \, ,
\]
where $\kcyl \hatotimes u_i$ is the Leibniz product of $\kcyl$ and $u_i \co A_i \to B_i$, as in~\eqref{equ:leib-prod}.
Now let $\cal{I}_\otimes \defeq \cal{I} + \cal{I}$ and define $u_\otimes \co \cal{I}_\otimes \to \cal{E}^\to$ via the coproduct diagram
\begin{equation} \label{equ:u-tensor}
\begin{gathered}
\xymatrix@C+2em{
  \cal{I}
  \ar[r]^{\iota_0}
  \ar[dr]_-{\lcyl \hatotimes u}
&
  \cal{I}_\otimes
  \ar[d]^(.4){u_\otimes}
&
  \cal{I}
  \ar[dl]^-{\rcyl \hatotimes u}
  \ar[l]_{\iota_1}
\\&
  \cal{E}^\to
\rlap{.}}
\end{gathered}
\end{equation}
We refer to the objects of $\cal{I}_\otimes$ as \emph{generating uniform trivial cofibrations}.

\begin{definition} \label{def:I-fibration} \leavevmode
\begin{enumerate}[(i)]
\item A \emph{uniform fibration} is a right $\cal{I}_\otimes$-map.
\item A \emph{morphism of uniform fibrations} is a morphism of right $\cal{I}_\otimes$-maps.
\end{enumerate}
\end{definition}

We write $\Fib$ for the category of uniform fibrations and their morphisms.
We now give the algebraic analogue of \cref{thm:char-fib-nonalg}, \ie an equivalent characterization of uniform fibrations which does not refer to generating uniform cofibrations.
This will show that the category $\Fib$ is independent from the choice of the functor $u \co \cal{I} \to \cal{E}^\to$ cofibrantly generating $(\CC, \TF)$, and instead depends only on the awfs $(\CC, \TF)$.

\begin{proposition} \label{prod-exp-general}
Uniform fibration structures on $p \in \cal{E}^\to$ are isomorphic to uniform trivial fibration structures on both $\hatexp(\bar{\delta}^0, p)$ and $\hatexp(\bar{\delta}^1, p)$, naturally in $p$.
More precisely, the category $\Fib$ is isomorphic over $\cal{E}^\to$ to the product of the categories $\cal{D}_0$ and $\cal{D}_1$ defined via pullbacks as follows:
\begin{align*}
\xymatrix@C+1em{
  \cal{D}_0
  \ar[r]
  \ar[d]
  \pullback{dr}
&
  \TrivFib
  \ar[d]
\\
  \cal{E}^\to
  \ar[r]_-{\hatexp(\bar{\delta}^0, p)}
&
  \cal{E}^\to
\rlap{,}}
&&
\xymatrix@C+1em{
  \cal{D}_1
  \ar[r]
  \ar[d]
  \pullback{dr}
&
  \TrivFib
  \ar[d]
\\
  \cal{E}^\to
  \ar[r]_-{\hatexp(\bar{\delta}^1, p)}
&
  \cal{E}^\to
\rlap{.}}
\end{align*}
\end{proposition}

\begin{proof}
Recall that $\TrivFib = \liftr{\cal{I}}$ and $\Fib = \liftr{(\cal{I}_\otimes)}$.
Note that the right orthogonality functor is contravariant and part of the adjunction~\eqref{garner-adjunction}, hence sends coproducts to products of categories over $\cal{E}^\to$.
The statement then follows from \cref{cor-to-most-gen-example} via~\cref{pitchfork-leibniz-most-general-applied}.
\end{proof}

We write~$\TrivCof$ for the category of uniform trivial cofibrations, defined as left $\Fib$-maps, \ie
\[
\Fib \defeq \liftr{(\cal{I}_\otimes)}
\,, \quad
\TrivCof \defeq \liftl{\Fib}
\,.\]
The next statement is immediate.

\begin{theorem} \label{thm:sset-cset-nwfs}
The category $\cal{E}$ admits a cofibrantly generated awfs $(\TC, \FF)$ such that
\[
\TrivCof = \pcoalg{\TC}
\,, \quad
\Fib = \pAlg{\FF}
\,.\]
\end{theorem}

\begin{proof}
By Garner's small object argument~\cite{garner:small-object-argument} (see also~\cite[Proposition~16]{bourke-garner-I}), applied to the generating small category of arrows $\cal{I}_\otimes$.
\end{proof}

Note again that there are functors over $\cal{E}^\to$ back and forth between $\palg{\FF}$ and $\pAlg{\FF}$ since $(\TC, \FF)$ is cofibrantly generated.
Let us remark that \cref{prod-exp-general} would not hold as stated above if we had defined uniform (trivial) fibrations as algebras for the pointed endofunctors of the relevant awfs's; instead of an isomorphism we would have only functors back and forth over~$\cal{E}^\to$.

\begin{example} \label{unif-fib-sset}
If we start from the awfs's $(\CC, \TF)$ in~$\SSet$ and~$\CSet$ of \cref{unif-triv-fib-sset} and apply~\cref{thm:sset-cset-nwfs}, we obtain awfs's $(\TC, \FF)$ whose monad algebras are algebraic counterparts of Kan fibrations, and hence will be called \emph{uniform Kan fibrations}.
In the case of $\CSet$, if we instead start with the awfs $(\CC', \TF')$ of \cref{unif-triv-fib-sset}, we obtain an awfs $(\TC', \FF')$ whose right maps are exactly the uniform Kan fibrations considered in~\cite{cohen-et-al:cubicaltt}.
An application to a different version of cubical sets (not using connections) gives the awfs considered in~\cite{swan-awfs}.
\end{example}

The next corollary, which will be useful to establish the functorial Frobenius property for~$(\TC, \FF)$ in \cref{sec:alg-frob}, is the algebraic counterpart of~\cref{thm:kcyl-of-cof-is-trivcof-non-alg}.

\begin{corollary} \label{kcyl-of-cof-is-trivcof}
Leibniz product with endpoint inclusions sends uniform cofibrations to uniform trivial cofibrations, in the sense that the Leibniz product functor lifts as follows for $k \in \braces{0\,,1}$:
\[
\xymatrix{
  \Cof
  \ar@{.>}[r]
  \ar[d]
&
  \TrivCof
  \ar[d]
\\
  \cal{E}^\to
  \ar[r]_-{\delta^k \hatotimes (-)}
&
  \cal{E}^\to
\rlap{.}}
\]
\end{corollary}

\begin{proof}
Recall that $\Cof = \liftl{\TrivFib}$ and $\TrivCof = \liftl{\Fib}$.
In the following diagram, the first lift is given by the map $\Fib \to \cal{D}_k$ of \cref{prod-exp-general} and the second lift is given by the unit of the adjunction~\eqref{garner-adjunction}:
\[
\xymatrix@C-1.5em{
  \Fib
  \ar@{.>}[rr]
  \ar[d]
&&
  \TrivFib
  \ar@{.>}[rr]
  \ar[dr]
&&
  \liftr{\Cof}
  \ar[dl]
\\
  \cal{E}^\to
  \ar[rrr]_-{\hatexp(\bar{\delta}^k, -)}
&&&
  \cal{E}^\to
\rlap{.}}
\]
The goal now follows from the outer square by applying \cref{pitchfork-leibniz-most-general-example} as in~\cref{pitchfork-leibniz-most-general-applied}.
\end{proof}

\section{The functorial Frobenius property for uniform fibrations}
\label{sec:alg-frob}

We now continue to consider the fixed suitable awfs $(\CC, \TF)$ with generating small category $\cal{I}$ and the induced awfs $(\TC, \FF)$, in which the monad algebras for $\FF$ are exactly the uniform fibrations.
Our aim in this section is to show that $(\TC, \FF)$ satisfies the functorial Frobenius property.

For this, we follow a strategy analogous to the one we used in \cref{sec:frob}.
We begin by organizing strong $k$-oriented homotopy equivalences into a category $\cal{S}_k$.
Its objects are 4-tuples $(f, g, \phi, \psi)$ consisting of an arrow $f \co A \to B$ together with data $g \co B \to A$, $\phi \co \interval \otimes A \to A$, $\psi \co \interval \otimes B \to B$ making $f$ into a strong $k$-oriented homotopy equivalence in the sense of~\cref{def:strhe}.
A morphism $m \co (f, g, \phi, \psi) \to (f', g', \phi', \psi')$ consists of maps $s \co A \to A', t \co B \to B'$ such that the following diagrams commute:
\begin{align*}
\xymatrix{
  A
  \ar[r]^{s}
  \ar[d]_{f}
&
  A'
  \ar[d]^{f'}
\\
  B
  \ar[r]_{t}
&
  B'
\rlap{,}}
&&
\xymatrix{
  B
  \ar[r]^{t}
  \ar[d]_{g}
&
  B'
  \ar[d]^{g'}
\\
  A
  \ar[r]_{s}
&
  A'
\rlap{,}}
&&
\xymatrix{
  \interval \otimes A
  \ar[d]_{\phi}
  \ar[r]^{\interval \otimes s}
&
  I \otimes A'
  \ar[d]^{\phi'}
\\
  A
  \ar[r]_{s}
&
  A'
\rlap{,}}
&&
\xymatrix{
  \interval \otimes B
  \ar[d]_{\psi}
  \ar[r]^{\interval \otimes t}
&
  I \otimes B'
  \ar[d]^{\psi'}
\\
  B
  \ar[r]_{t}
&
  B'
\rlap{.}}
\end{align*}
There is an obvious first projection functor $p_k \co \cal{S}_k \to \cal{E}^\to$.
\cref{strong-h-equiv-as-section} below extends the logical equivalence of \cref{strong-h-equiv-as-section-non-alg} to an isomorphism of categories.
In its statement, we refer to the maps~$\thetak \hatotimes f \co f \to \kcyl \hatotimes f$ in~\eqref{equ:thetak}.

\begin{lemma} \label{strong-h-equiv-as-section}
For $k \in \braces{0\,, 1}$, the category $\cal{S}_k$ of strong $k$-oriented homotopy equivalences can be described isomorphically as the category of arrows $f \in \cal{E}^{\to}$ with a retraction $\rho$ of $\thetak \hatotimes f$.
In detail,
\begin{enumerate}[(i)]
\item objects are pairs $(f, \rho)$ consisting of $f \in \cal{E}^\to$ and a retraction $\rho$ of $\thetak \hatotimes f$, as below:
\[
\xymatrix@C+2em{
  f
  \ar[r]^-{\thetak \hatotimes f}
  \ar@{=}[dr]
&
  \kcyl \hatotimes f \ar[d]^{\rho}
\\&
  f
\rlap{,}}
\]
\item morphisms $\tau \co (f, \rho) \to (f', \rho')$ are maps $\tau \co f \to f'$ such that the below diagram commutes:
\[
\xymatrix@C+2em{
  \kcyl \hatotimes f
  \ar[d]_-{\rho}
  \ar[r]^{\kcyl \hatotimes \tau}
&
 \kcyl \hatotimes f'
  \ar[d]^-{\rho'}
\\
  f
  \ar[r]_{\tau}
&
  f'
\rlap{.}}
\]
\end{enumerate}
\end{lemma}

\begin{proof}
The object part of the correspondence is essentially \cref{strong-h-equiv-as-section-non-alg}.
The morphism part is straightforward.
\end{proof}

\begin{remark} \label{theta}
As suggested by our notation, the maps~$\thetak \hatotimes f \co f \to \kcyl \hatotimes f$ are really given as the action on objects of a natural transformation $\thetak \hatotimes (-) \co \Id \to \delta^k \hatotimes (-)$ as follows.
Let $\theta^k \co \bot_{\Id} \to \delta^k \otimes (-)$ denote the arrow in $[\cal{E}, \cal{E}^\to] \iso [\cal{E}, \cal{E}]^\to$ given by the following square:
\[
\xymatrix@C+3em{
  0
  \ar[r]^-{\bot_{\Id}}
  \ar[d]_-{\bot_{\Id}}
&
  \Id
  \ar[d]^-{\delta^k \otimes (-)}
\\
  \Id
  \ar[r]_-{\delta^{1-k} \otimes (-)}
&
  \interval \otimes (-)
\rlap{.}}
\]
Writing $(-) \hatotimes (-) \co [\cal{E}, \cal{E}]^\to \times \cal{E}^\to \to \cal{E}^\to$ for the functor obtained by applying the Leibniz construction to the evaluation functor (\cf \cref{pitchfork-leibniz-most-general-applied}), the desired natural transformation is precisely~$\theta^k \hatotimes (-)$.
\end{remark}

\begin{remark} \label{thm:endpoint-are-she-functorially}
Recall from \cref{thm:endpoint-are-she} that the components of the endpoint inclusion $\kcyl \otimes (-)$ are strong $k$-oriented homotopy equivalences.
In fact, this is the action on objects of a factorization of $\kcyl \otimes (-)$ through $\cal{S}_k$:
\[
\xymatrix@C+3em{
&
  \cal{S}_k
  \ar[d]
\\
  \cal{E}
  \ar[r]_{\kcyl \otimes (-)}
  \ar[ur]^{\widetilde{\kcyl \otimes (-)}}
&
  \cal{E}^\to
\rlap{.}}
\]
\end{remark}

The next lemma is the analogue of \cref{thm:kcylf-is-she}.

\begin{lemma} \label{thm:kcylf-is-she-alg}
For $k \in \braces{0\,, 1}$, the functor $\kcyl \hatotimes (-) \co \cal{E}^\to \to \cal{E}^\to$ factors via $\cal{S}_k$ as follows:
\[
\xymatrix@C+3em{
&
  \cal{S}_k
  \ar[d]
\\
  \cal{E}^\to
  \ar[r]_-{\kcyl \hatotimes (-)}
  \ar[ur]^-{\widetilde{\kcyl \hatotimes (-)}\,}
&
  \cal{E}^\to
\rlap{.}}
\]
\end{lemma}

\begin{proof}
We give a proof using the formalism of Leibniz constructions, as explained in~\cite{riehl-verity:reedy}.
For this, we use the description of $\cal{S}_k$ provided by \cref{strong-h-equiv-as-section} and the notation of \cref{theta}, thus writing simply $\delta^k$ for the natural transformation $\delta^k \otimes (-)$.
We apply the Leibniz construction to the endofunctor composition functor $(-) \cc (-) \co [\cal{E}, \cal{E}] \times [\cal{E}, \cal{E}] \to [\cal{E}, \cal{E}]$, so as to obtain a functor
\[
(-) \mathbin{\hat{\cc}} (-) \co [\cal{E}, \cal{E}]^\to \times [\cal{E}, \cal{E}]^\to \to [\cal{E}, \cal{E}]^\to
.\]

Recall from \cref{thm:endpoint-are-she-functorially} that $\delta^k \co \cal{E} \to \cal{E}^\to$ factors via $\cal{S}_k \to \cal{E}^\to$, \ie that the arrow
\[
\theta^k \hatotimes (\delta^k \otimes X) \co \delta^k \otimes X \to \delta^k \hatotimes (\delta^k \otimes X)
\]
exhibits $\delta^k \otimes X$ as a retract of $\delta^k \hatotimes (\delta^k \otimes X)$, functorially in $X \in \cal{E}$.
This amounts to saying that the map of natural transformations
\[
\theta^k \mathbin{\hat{\cc}} \delta^k \co \delta^k \to \delta^k \mathbin{\hat{\cc}} \delta^k
\]
exhibits $\delta^k$ as a retract of $\delta^k \mathbin{\hat{\cc}} \delta^k$ in $[\cal{E}, \cal{E}]^\to$.

Our goal is to show that $\delta^k \hatotimes (-) \co \cal{E}^\to \to \cal{E}^\to$ can be factored via $\cal{S}_k \to \cal{E}^\to$.
For this, we need to show that the arrow
\[
\theta^k \hatotimes (\delta^k \hatotimes f) \co \delta^k \hatotimes f \to \delta^k \hatotimes (\delta^k \hatotimes f)
\]
exhibits $\delta^k \hatotimes f$ as a retract of $\delta^k \hatotimes (\delta^k \hatotimes f)$, functorially in $f \in \cal{E}^\to$.
But associativity of functor composition implies associativity of the corresponding Leibniz construction (since the involved functors preserve pushouts), and therefore we can restate our goal as asserting that the map
\[
(\theta^k \mathbin{\hat{\cc}} \delta^k) \hatotimes f \co \delta^k \hatotimes f \to (\delta^k \mathbin{\hat{\cc}} \delta^k) \hatotimes f
\]
exhibits $\delta^k \hatotimes f$ as a retract of $(\delta^k \mathbin{\hat{\cc}} \delta^k) \hatotimes f$, functorially in $f \in \cal{E}^\to$.
But functors, in this case the functor $(-) \hatotimes (-) \co [\cal{E}, \cal{E}]^\to \times \cal{E}^\to \to \cal{E}^\to$ mapping $(\alpha, f)$ to $\alpha \hatotimes f$, preserve section-retraction pairs.
\end{proof}

The next lemma is the analogue of~\cref{thm:main-sheretract}.
Define $\cal{S} \defeq \cal{S}_0 + \cal{S}_1$.

\begin{proposition} \label{thm:strong-hequiv} \leavevmode
\begin{enumerate}[(i)]
\item \label{thm:onedir} There is a functor
\[
\xymatrix@!C@C-2.5em{
  \cal{I}_\otimes
  \ar[dr]_{u_\otimes}
  \ar[rr]^F
&&
  \Cof \times_{\cal{E}^\to} \cal{S}
  \ar[dl]
\\&
  \cal{E}^\to
\rlap{.}}
\]
\item \label{thm:twodir} There is a functor
\[
\xymatrix@!C@C-2.5em{
  \Cof \times_{\cal{E}^\to} \cal{S}
  \ar[dr]
  \ar@<4pt>[rr]^{H}
&&
  \TrivCof
  \ar[dl]
\\&
  \cal{E}^\to
\rlap{.}}
\]
\end{enumerate}
\end{proposition}

\begin{proof}
We begin by considering~(i).
For $k \in \braces{0 \, , 1}$, we show that there is a functor
\[
\xymatrix@!C@C-2em{
  \cal{I}
  \ar[dr]_{\kcyl \hatotimes u} \ar[rr]^{M_k}
&&
  \Cof \times_{\cal{E}^\to} \cal{S}_k
  \ar[dl]
\\&
   \cal{E}^\to
\rlap{.}}
\]
The map to the first factor is given by the embedding $\cal{I} \to \Cof$ and the assumption that $\Cof$ is closed under Leibniz product with the endpoint inclusions, which is condition~(S4) of \cref{thm:suitable-awfs}.
The map to the second factor is given by \cref{thm:kcylf-is-she-alg}.
The claim in~(i) then follows by combining the cases~$k = 0$ and~$k = 1$.

For~(ii), let us write $c \co \Cof \to \cal{E}^\to$ for the forgetful functor.
We fix again $k \in \braces{0\,, 1}$ and show that there is a functor
\begin{equation} \label{lem:from-strong-hequiv}
\begin{gathered}
\xymatrix@!C@C-2em{
  \Cof \times_{\cal{E}^\to} \cal{S}_k
  \ar[dr]
  \ar[rr]^{N_k}
&&
  \overline{\Cof}
  \ar[dl]^-{\overline{\kcyl \hatotimes c}}
\\&
  \cal{E}^\to
\rlap{.}}
\end{gathered}
\end{equation}
We only describe the action of the functor $N_k$ on an object $(i, \rho)$ with $i \in \Cof$ and $\rho$ a retraction of $\thetak \hatotimes c_i \co c_i \to \kcyl \hatotimes c_i$, leaving the evident definition of the action on arrows to the reader.
Since~$\rho$ exhibits~$c_i$ as a retract of~$\kcyl \hatotimes c_i$, we may define $N_k(i, \rho) \defeq (i, c_i, \thetak \hatotimes c_i, \rho)$.
Observe that this definition makes the diagram for $N_k$ commute.
We then pass from $\overline{\kcyl \hatotimes c} \co \overline{\Cof} \to \cal{E}^\to$ to $\overline{\TrivCof}$ using \cref{kcyl-of-cof-is-trivcof}.
Upon combining the cases $k = 0$ and $k = 1$, this gives a functor $\Cof \times_{\cal{E}^\to} \cal{S} \to \overline{\TrivCof}$ over $\cal{E}^\to$.
The claim in~(ii) then follows by \cref{retract-closure} and the adjunction~\eqref{garner-adjunction}, noting that $\TrivCof$ is a category of left maps.
\end{proof}

\begin{remark} \label{relating-strong-hequiv-and-uniform-fib}
We also have a pair of functors $\liftr{(\Cof \times_{\cal{E}^\to} \cal{S})} \leftrightarrow \Fib$ relating uniform fibrations with right $(\Cof \times_{\cal{E}^\to} \cal{S})$-maps.
This means that a map can be equipped with the structure of a uniform fibration if and only if it can be equipped with the structure of a right $(\Cof \times_{\cal{E}^\to} \cal{S})$-map.
This is the analogue of the equality between sets in \cref{fib-and-she}.
However, these functors do not in general form an equivalence.
\end{remark}

We are now ready to show that the awfs in which the right maps are the uniform fibrations (constructed in~\cref{thm:sset-cset-nwfs}) satisfies the functorial Frobenius property.
For this, we use our characterisation of the functorial Frobenius property in algebraically-free awfs's stated in \cref{thm:frobenius-comparison}.
As the main step, we show that that strong homotopy equivalences and uniform fibrations satisfy the generalized functorial Frobenius condition, \ie the algebraic analogue of \cref{thm:non-alg-frobenius-she}.

\begin{lemma} \label{technical}
The triple $(\cal{S}, \Fib, \cal{S})$ satisfies the generalized functorial Frobenius condition.
\end{lemma}

\begin{proof}
The core of the argument is to obtain, for $k \in \braces{0\,, 1}$, a lift $\tilde{P}$ in
\begin{equation} \label{strong-h-equiv-uniform-base-change:goal}
\begin{gathered}
\xymatrix@C+2em{
  \cal{S}_k \times_{\cal{E}} \Fib
  \ar@{.>}[r]^{\tilde{P}}
  \ar[d]
&
  \cal{S}_k
  \ar[d]
\\
  \cal{E}^\to \times_{\cal{E}} \cal{E}^\to
  \ar[r]_{P}
&
  \cal{E}^\to
}
\end{gathered}
\end{equation}
where $P \co \cal{E}^\to \times_{\cal{E}} \cal{E}^\to \to \cal{E}^\to$ is the pullback functor, sending a cospan $(g, h)$ to $h^* g$.
The goal then follows from \cref{generalized-uniform-frob-nicer} upon combining the cases $k = 0$ and $k = 1$ using \cref{generalized-uniform-frob-product-u}.

So it remains to show how to obtain the lift in \eqref{strong-h-equiv-uniform-base-change:goal}.
Its action on objects is described in the proof of~\cref{thm:non-alg-frobenius-she}.
For the lift of the action on morphisms, suppose we are given a map $j \to j'$ in $\Fib$, inducing a square
\[
\xymatrix{
  X
  \ar[r]^{v_j}
  \ar[d]_{s}
&
  Y
  \ar[d]^{t}
\\
  X'
  \ar[r]_{v_{j'}}
&
  Y'
\rlap{,}}
\]
and a map $\tau \co (g, \rho) \to (g', \rho')$ in $\cal{S}_k$, with $\tau$ forming a square
\[
\xymatrix{
  B
  \ar[r]
  \ar[d]_{g}
&
  B'
  \ar[d]^{g'}
\\
  Y
  \ar[r]_-{t}
&
  Y'
}
\]
and commuting with the retractions $\rho$ and $\rho'$ as follows:
\[
\xymatrix{
  \kcyl \hatotimes g
  \ar[r]^-{\rho}
  \ar[d]_{\kcyl \hatotimes \tau}
&
  g
  \ar[d]^{\tau}
\\
  \kcyl \hatotimes g'
  \ar[r]_-{\rho'}
&
  g'
\rlap{.}}
\]
Let $(\bar{g}, \bar{\rho})$ and $(\bar{g}', \bar{\rho}')$ denote the respective action of $\tilde{P}$ on the objects $(g, \rho, j)$ and $(g', \rho', j')$ as constructed in \cref{thm:non-alg-frobenius-she}.
Recall that this includes pullback squares $\sigma \co \bar{g} \to g$ and $\sigma' \co \bar{g}' \to g'$ with bottom side $v_j \co X \to Y$ and $v_{j'} \co X' \to Y'$, respectively, as in~\eqref{non-alg-strong-h-equiv-uniform-base-change:0}.
The square $\tau \co g \to g'$ pulls back to a square $\bar{\tau} \co \bar{g} \to \bar{g'}$ with bottom side $v_{j'}$.
We want to show that $\bar{\tau}$ in addition forms a morphism of strong $k$-oriented homotopy equivalences from $(\bar{g}, \bar{\rho})$ to $(\bar{g}', \bar{\rho}')$.
For this, we have to verify commutativity of the following diagram:
\[
\xymatrix{
  \kcyl \hatotimes \bar{g}'
  \ar[r]^-{\bar{\rho}}
  \ar[d]_{\kcyl \hatotimes \bar{\tau}}
&
  \bar{g}
  \ar[d]^{\bar{\tau}}
\\
  \kcyl \hatotimes \bar{g}'
  \ar[r]_-{\bar{\rho}'}
&
  \bar{g}'
\rlap{.}}
\]
Recall the construction of $\bar{\rho}$ and $\bar{\rho}'$, omitting horizontal composite identities for readability:
\[
\xymatrix@!C{
  \bar{g}
  \ar[rr]^-{\thetak \hatotimes \bar{g}}
  \ar[dd]_{\sigma}
  \ar[dr]^{\bar{\tau}}
&&
  \kcyl \hatotimes \bar{g}
  \ar@{.>}[rr]^-{\bar{\rho}}
  \ar[dd]^(0.3){\kcyl \hatotimes \sigma}
  \ar[dr]^{\kcyl \hatotimes \bar{\tau}}
&&
  \bar{g}
  \ar[dd]^(0.3){\sigma}
  \ar[dr]^{\bar{\tau}}
\\&
  \bar{g}'
  \ar[rr]^-(0.3){\thetak \hatotimes \bar{g}'}
  \ar[dd]_(0.3){\sigma'}
&&
  \kcyl \hatotimes \bar{g}'
  \ar@{.>}[rr]^-(0.3){\bar{\rho}'}
  \ar[dd]^(0.3){\kcyl \hatotimes \sigma'}
&&
  \bar{g}'
  \ar[dd]^{\sigma'}
\\
  g
  \ar[rr]^-(0.25){\thetak \hatotimes g}
  \ar[dr]^{\tau}
&&
  \kcyl \hatotimes g
  \ar[rr]^-(0.3){\rho}
  \ar[dr]^{\kcyl \hatotimes \tau}
&&
  g
  \ar[dr]^{\tau}
\\&
  g'
  \ar[rr]^-{\thetak \hatotimes g'}
&&
  \kcyl \hatotimes g
  \ar[rr]^-{\rho'}
&&
  g
\rlap{.}}
\]
Our goal is to show that the top right square commutes.
Since that square commutes after composing it with the pullback square $\sigma'$, it suffices to show that the square commutes when projected to codomains, again omitting horizontal composite identities:
\[
\xymatrix@C+2em{
  X
  \ar[rr]^-{\kcylinv \otimes X}
  \ar[dd]_{v_j}
  \ar[dr]^{s}
&&
  \interval \otimes X
  \ar@{.>}[rr]^-{\cod(\bar{\rho})}
  \ar[dd]^(0.35){\interval \otimes v_j}
  \ar[dr]^{\interval \otimes s}
&&
  X
  \ar[dd]^(0.35){v_j}
  \ar[dr]^{s}
\\&
  X'
  \ar[rr]^-(0.3){\kcylinv \otimes X'}
  \ar[dd]^(0.3){v_{j'}}
&&
  \interval \otimes X'
  \ar@{.>}[rr]^-(0.3){\cod(\bar{\rho}')}
  \ar[dd]^(0.3){\interval \otimes v_{j'}}
&&
  X'
  \ar[dd]^{v_{j'}}
\\
  Y
  \ar[rr]^-(0.25){\kcylinv \otimes Y}
  \ar[dr]^{t}
&&
  \interval \otimes Y
  \ar[rr]^-(0.3){\cod(\rho)}
  \ar[dr]^{\interval \otimes t}
&&
  Y
  \ar[dr]^{t}
\\&
  Y'
  \ar[rr]^-{\kcylinv \otimes Y'}
&&
  \interval \otimes Y'
  \ar[rr]^-{\cod(\rho')}
&&
  Y'
\rlap{.}}
\]
But this follows from coherence of lifts in the following morphism of lifting problems:
\[
\xymatrix@C+1em{
  X
  \ar@{=}[rrrr]
  \ar[dd]_{\kcylinv \otimes X}
  \ar[dr]^{s}
&&&&
  X
  \ar[dd]_(0.3){v_j}
  \ar[dr]^{s}
\\&
  X'
  \ar@{=}[rrrr]
  \ar[dd]^(0.66){\kcylinv \otimes X'}
&&&&
  X'
  \ar[dd]^{v_{j'}}
\\
  \interval \otimes X
  \ar[rr]^(0.7){I \otimes v_j}
  \ar@{.>}[uurrrr]^(0.67){\cod(\bar{\rho})}
  \ar[dr]_{I \otimes s}
&&
  \interval \otimes Y
  \ar[rr]_(0.7){\cod(\rho)}
  \ar[dr]^(0.7){\interval \otimes t}
&&
  Y
  \ar[dr]^{t}
\\&
  \interval \otimes X'
  \ar[rr]_{I \otimes v_{j'}}
  \ar@{.>}[uurrrr]^(0.67){\cod(\bar{\rho}')}
&&
  \interval \otimes Y'
  \ar[rr]_{\cod(\rho')}
&&
  Y'
\rlap{.}}
\]
Here, the left and right faces form morphisms in $\TrivCof$ (using condition~(S2) of \cref{thm:suitable-awfs,kcyl-of-cof-is-trivcof}) and $\Fib$, respectively, making the lifts cohere as needed.
\end{proof}

\begin{theorem} \label{uniform-fibrations-uniform-frobenius}
The awfs $(\mathsf{C_t}, \mathsf{F})$ satisfies the functorial Frobenius condition.
\end{theorem}

\begin{proof}
First, we show that $(\Cof \times_{\cal{E}} \cal{S}, \Fib, \Cof \times_{\cal{E}} \cal{S})$ satisfies the generalized functorial Frobenius condition.
This follows from \cref{generalized-uniform-frob-product-u}: by condition~(S3) in the definition of \cref{thm:suitable-awfs} and \cref{generalized-uniform-frob-nicer}, $(\Cof, \cal{E}^\to, \Cof)$ satisfies the generalized functorial Frobenius condition, hence also $(\Cof, \Fib, \Cof)$ by \cref{generalized-uniform-frob-functorial}; while $(\cal{S}, \Fib, \cal{S})$ satisfies the generalized Frobenius condition by \cref{technical}.

The generalized functorial Frobenius condition for $(\Cof \times_{\cal{E}} \cal{S}, \Fib, \Cof \times_{\cal{E}} \cal{S})$ implies the generalized functorial Frobenius condition for $(\cal{I}_\otimes, \Fib, \TrivCof)$ using functoriality of generalized functorial Frobenius in the form of \cref{generalized-uniform-frob-functorial} with $F$ the functor in~\eqref{thm:onedir} of~\cref{thm:strong-hequiv}, $G$ the identity, and $H$ the functor in~\eqref{thm:twodir} of~\cref{thm:strong-hequiv}.
This implies the functorial Frobenius condition for $(\mathsf{C_t}, \mathsf{F})$ by \cref{thm:frobenius-comparison}.
\end{proof}

As special cases, we obtain the pushforward versions of the Frobenius and Beck-Chevalley condition for uniform fibrations.
First, pushforward lifts to slices of the category of uniform $\cal{I}$-fibrations.

\begin{corollary} \label{uniform-fibrations-frobenius-pushforward} \leavevmode
\begin{enumerate}[(i)]
\item For every uniform fibration
$p \co X \to Y$, pushforward along $p$ lifts to a functor
\[
\xymatrix@C+2em{
  \Fib_{/X}
  \ar[r]^{\widetilde{p_*}}
  \ar[d]
&
  \Fib_{/Y}
  \ar[d]
\\
  \cal{E}_{/X}^\to
  \ar[r]_{p_*}
&
  \cal{E}_{/Y}^\to
\rlap{.}}
\]
\item For every map of uniform $\cal{I}$-fibrations $(s, t) \co p \to q$, where $p \co X \to Y$ and $q \co U \to V$, the canonical natural transformation $\psi \co t^* q_* \to p_* s^*$ lifts to a natural transformation
\[
\xymatrix@C+2em{
  \Fib_{/U}
  \ar[r]^{\widetilde{q_*}}
  \ar[d]_{\widetilde{s^*}}
  \ar@{}[dr]|{\textstyle\Downarrow \rlap{$\labelstyle \widetilde{\psi}$}}
&
  \Fib_{/V}
  \ar[d]^{\widetilde{t^*}}
\\
  \Fib_{/X}
  \ar[r]_{\widetilde{p_*}}
&
  \Fib_{/Y}
\rlap{.}}
\]
If $(s, t) \co p \to q$ forms a pullback square, then $\widetilde{\psi}$ is a natural isomorphism.
\end{enumerate}
\end{corollary}

\begin{proof}
The claim follows from \cref{uniform-fibrations-uniform-frobenius} using \cref{thm:frobenius-comparison,lift-pushforward}.
\end{proof}

\begin{example}
The application of \cref{uniform-fibrations-frobenius-pushforward} in $\SSet$ and $\CSet$ shows that pushforward along a uniform Kan fibration preserves uniform Kan fibrations.
Since exponentiation is a special case of pushforward, this result shows also that if $X$ is a uniform Kan complex (defined in the evident way) then $Y^X$ is again a uniform Kan complex, for every $Y$.
In fact, for this easier result, only the base needs to be assumed uniform Kan.
It follows by setting $A \defeq 1$ and $B \defeq 0$ from the general statement that the Leibniz exponential of a uniform Kan fibration $p \co X \to A$ with a cofibration $i \co B \to Y$ gives a uniform Kan fibration $\hatexp(i, p)$.

As is well known, this can be seen as follows.
Let $k \in \braces{0\,, 1}$.
By \cref{prod-exp-general}, $\hatexp(\bar{\delta}^k, p)$ is a uniform trivial Kan fibration.
Since uniform cofibrations are closed under Leibniz product with $i$, by adjointness $\hatexp(i, \hatexp(\bar{\delta}^k, p)) = \hatexp(\bar{\delta}^k, \hatexp(i, p))$ is a uniform trivial Kan fibration.
Combining the cases $k = 0$ and $k = 1$ and using \cref{prod-exp-general} in the reverse direction, it follows that $\hatexp(i, p)$ is a uniform Kan fibration.
\end{example}

\section{Uniform fibrations in presheaf categories}
\label{sec:fib-psh}

The aim of this final section is to study in more detail the notion of a uniform fibration in the case when $\cal{E}$ is a presheaf category.
Let us begin by fixing the setting in which we shall be working.
First, let $\mathbb{C}$ be a small category and fix $\cal{E} \defeq \Psh(\mathbb{C})$.
We write $\yon \co \mathbb{C} \to \cal{E}$ for the Yoneda embedding.
We assume that the functorial cylinder~$\interval \otimes (-) \co \cal{E} \to \cal{E}$ satisfies not only our standing assumptions of having contractions and connections and of possessing a right adjoint (see \cref{sec:bac}), but also the following two conditions:
\begin{enumerate}[({C}1)]
\item $\interval \otimes (-) \co \cal{E} \to \cal{E}$ preserves pullback squares,
\item the natural transformations $\kcyl \otimes (-) \co \Id_{\cal{E}} \to \interval \otimes (-)$, for $k \in \braces{0\,, 1}$, are cartesian.
\end{enumerate}
Second, we let $\cal{M}$ be a full subcategory of $\cal{E}^\to_\cart$ satisfying the following assumptions:
\begin{enumerate}[({M}1)]
\item the elements of $\cal{M}$ are monomorphisms,
\item for every $X \in \cal{E}$, the unique map $\bot_X \co 0 \to X$ is in $\cal{M}$,
\item the elements of $\cal{M}$ are closed under pullback,
\item the elements of $\cal{M}$ are closed under Leibniz product with the endpoint inclusions,
\item the category $\cal{M}$ has colimits and the inclusion $\cal{M} \to \cal{E}^\to_\mathrm{cart}$ preserves them.%
\footnote{This assumption is not used to instantiate our general results, but only for~\cref{partial-map-classifier,unif-vs-non-unif}.}
\end{enumerate}
The next result not only provides us with a wide class of examples of suitable awfs's, but also shows that there are situations in which object-wise assertions, as in (M2)-(M4) above, can be strengthened to functoriality properties, as required to obtain a suitable weak awfs (\cref{thm:suitable-awfs}).

\begin{theorem} \label{rem-lift-suitable}
There exists a suitable awfs $(\CC, \TF)$ on $\cal{E}$ that is algebraically-free on $\cal{M}$, \ie such that $\TrivFib = \liftr{\cal{M}}$.
\end{theorem}

\begin{proof}
Let $\cal{I}$ be the full subcategory of $\cal{M}$ spanned by maps with a representable presheaf as codomain.
Since $\cal{I}$ is small, by Garner's small object argument~\cite{garner:small-object-argument} it generates a cofibrantly generated awfs $(\CC, \TF)$ with~$\TrivFib = \pAlg{\TF} = \liftr{\cal{I}}$.
Furthermore, we have~$\liftr{\cal{M}} = \liftr{\cal{I}}$.
This follows from \cref{awfs-on-arrows-into-representables} using that $\cal{M}$ is a full subcategory of~$\cal{E}^\to_\cart$.
Indeed, for a map $f \co X \to Y$ in $\cal{E}$, to give a natural choice of fillers for all diagrams with an arbitrary element of~$\cal{M}$ on the left is the same as to give a natural choice of fillers for only those diagrams which are the form
\[
\xymatrix{
  A
  \ar[r]
  \ar[d]
&
  X
  \ar[d]^f
\\
  \yon(x)
  \ar[r]
&
  Y
}
\]
for $x \in \catC$.

It remains to check that $(\CC, \TF)$ satisfies conditions~(S2)-(S4) in~\cref{thm:suitable-awfs}.
These follow from the corresponding object-wise assumptions~(M2)-(M4) via standard diagram-chasing arguments, which we omit.
For condition~(S4), the crucial property being used in that proof is that elements of $\cal{M}$, being monomorphisms in a topos, are adhesive maps~\cite{garner-lack:adhesive}.
\end{proof}

\begin{remark} \label{rem:colimit-decomp}
Let us point out that the reduction to the lifting problems against maps in $\cal{M}$ to those with a representable codomain in the proof of \cref{rem-lift-suitable} exploits the good behaviour of the right orthogonality functor with respect to colimits as described in~\cref{awfs-on-arrows-into-representables}, which is not available in the setting of ordinary wfs's.
\end{remark}

\begin{example} \label{justify-sset-cset-examples}
Since assumptions (C1)-(C2) and (M1)-(M5) are fulfilled in both $\SSet$ and~$\CSet$, with $\cal{M}$ the category of all monomorphisms and pullback squares (with (M5) a consequence of the existence of a subobject classifier), \cref{rem-lift-suitable} establishes the existence of the awfs $(\TC, \FF)$ of \cref{unif-fib-sset}.
The assumptions are also satisfied if instead in $\CSet$ we choose for $\cal{M}$ the category $\cal{M}'$ of \cref{unif-triv-fib-sset} (with elements classified by the face lattice $\Phi$ of~\cite{cohen-et-al:cubicaltt}), establishing the existence of the awfs $(\TC', \FF')$ of \cref{unif-fib-sset}.
\end{example}

\begin{remark} \label{rem:constructive-small-object}
We discuss how, in the case of $\SSet$ and~$\CSet$, the existence of the awfs's $(\CC, \TF)$, having uniform trivial Kan fibrations as monad algebras, and $(\TC, \FF)$, having uniform Kan fibrations as monad algebras, and the Frobenius property for $(\TC, \FF)$ can be proved constructively, \ie without using the law of excluded middle or the axiom of choice.

We begin by observing that \cref{rem-lift-suitable} can be proved constructively in $\SSet$ and $\CSet$.
In order to see why this is the case, first note that in these examples every subobject of a representable is finitely presentable and that the functorial cylinder $\interval \otimes (-) \co \cal{E} \to \cal{E}$ preserves finitely presentable objects.
Since subobjects of representables are finitely presentable, the values of the inclusion $u \co \cal{I} \to \cal{E}^\to$ are finitely presentable objects of $\cal{E}^\to$.
An inspection of the proof of \cite[Theorem~4.4]{garner:small-object-argument} shows that this suffices to construct the algebraically-free awfs $(\CC, \TF)$ on $u \co \cal{I} \to \cal{E}^\to$, and in fact the sequence constructing the appropriate free monad converges after $\omega$ steps.
Next, note that also \cref{thm:sset-cset-nwfs} can be proved constructively for $\SSet$ and $\CSet$, via a reasoning that is analogous to the one above, observing that also the values of $u_\otimes \co \cal{I}_\otimes \to \cal{E}^\to$ are finitely presentable by the assumption of the functorial cylinder and the fact that finitely presentable objects are closed under pushout.
Finally, the general proofs of \cref{uniform-fibrations-uniform-frobenius,uniform-fibrations-frobenius-pushforward}, establishing the functorial Frobenius condition for $(\TC, \FF)$ and its pushforward analogue, are constructive.

We have therefore obtained a constructive proof that pushforward along a uniform Kan fibration preserves uniform Kan fibrations in $\SSet$ and $\CSet$.
For $\SSet$, this result is a constructive counterpart of the fact that pushforward along a Kan fibration preserves Kan fibrations, which cannot be proved constructively without changes to the definition of Kan fibration~\cite{coquand-non-constructivity-kan}.

Finally, let us point out that if one adds the assumption that elements of $\cal{M}$ are \emph{decidable} monomorphisms, then the argument above carries over without relying on the Power Set axiom to establish the smallness of $\cal{I}$.
%We believe that this restriction is important also to treat constructively further aspects of the theory, such as the existence of fibrant universes~\cite{cohen-et-al:cubicaltt}.
\end{remark}

The rest of this section is devoted to giving a characterization of uniform trivial fibrations in terms of the partial map classifier, a result suggested to us by Thierry Coquand and Andr\'e Joyal (see also \cref{rem:bg} below).
We then use this characterization to relate non-algebraic and algebraic notions of (trivial) fibration.

\begin{remark} \label{rem:bg}
A related, but different, approach to generating an awfs $(\CC, \TF)$ from a class of monomorphisms is taken in \cite{bourke-garner-I}.
Recall from \cref{rem-lift-suitable} that $(\CC, \TF)$ is cofibrantly generated by the full subcategory of $\cal{E}^\to_\cart$ spanned by monomorphisms, written $\cal{M}$ below.
Instead, as done in \cite{bourke-garner-I}, one may assume that the chosen monomorphisms are closed under composition and view $\cal{M}$ as a double category.
This then cofibrantly generates an awfs $(\CC', \TF')$ in a double categorical sense.
The monad $\TF'$ in the slice over $Y \in \cal{E}$ is given by the $\cal{M}$-partial map classifier $P_Y$ in the slice over $Y$ defined before \cref{thm:part-map-classifier} below.
For us instead, as implies by \cref{thm:part-map-classifier}, the monad $\TF$ in the slice over $Y \in \cal{E}$ is freely generated by $P_Y$ seen as a pointed endofunctor.
It follows that the awfs's $(\CC, \TF)$ and $(\CC', \TF')$ are different, but related by an isomorphism $\pAlg{\TF} \cong \palg{\TF'}$.
Since $\delta^k \hatotimes (-)$ for $k \in \braces{0\,,1}$ do not lift to endofunctors on the double category of squares in $\cal{E}$, it is not clear how to construct a double categorical version of the generating uniform trivial cofibrations $\cal{I}_\otimes$ to generate a corresponding awfs $(\TC', \FF')$.
\end{remark}

We begin with a simple lemma.

\begin{lemma} \label{identities-in-M}
For all $X \in \cal{E}$, $\id_X \in \cal{M}$.
\end{lemma}

\begin{proof}
The claim follows by inspection of the following diagram:
\[
\xymatrix@C+3em{
  X
  \ar[r]^{\id_X}
  \ar[d]_{\id_X}
  \pullback{dr}
&
  \Id \otimes X
  \ar[d]^{\delta^0 \hatotimes \bot_X}
\\
  \Id \otimes X
  \ar[r]_{\delta^0 \hatotimes \bot_X}
&
  I \otimes X
\rlap{.}}
\]
By conditions~(M2) and~(M4), the right map is in $\cal{M}$.
By condition~(C2), the square is a pullback.
By condition~(M3), the left map $\id_X$ is then in $\cal{M}$.
\end{proof}

Below, as usual, we write $\Omega$ for the subobject classifier of the topos $\cal{E}$ and $\top \co 1 \to \Omega$ for its `true' morphism.

\begin{lemma} \label{partial-map-classifier}
There exists a subobject $K \rightarrowtail \Omega$ and a factorization
\[
\xymatrix{
  1
  \ar[rr]^\top
  \ar[dr]
&&
  \Omega
\\
&
  K
  \ar@{>->}[ur]
}
\]
such that the map $1 \to K$ classifies maps in~$\cal{M}$, in the sense that a map $i \co A \to B$ is in $\cal{M}$ if and only if there exists a pullback
\[
\xymatrix{
  A
  \ar[r]
  \ar[d]_i
  \drpullback
&
  1
  \ar[d]
\\
  B
  \ar[r]
&
  K
\rlap{.}}
\]
\end{lemma}

\begin{proof}
For $x \in \cat{C}$, define $K(x)$ to be the set of subobjects of~$\yon(x)$ that are elements of~$\cal{M}$.
The factorisation of $\top \co 1 \to \Omega$ through $K$ follows since identities are contained in $\cal{M}$ by \cref{identities-in-M}.

Any element $i \co A \to B$ of $\cal{M}$ arises as base change of $1 \to K$ along a unique map $B \to K$, sending an element $b \co \yon(x) \to B$, for $x \in \cat{C}$, to the subobject of $\yon(x)$ given by the base change of~$i \co A \to B$ along $b$, which is in $K(x)$ since $i$ is in $\cal{M}$.
For the converse, it suffices to show that~$1 \to K$ is in $\cal{M}$, but this is equivalent to (M5).
\end{proof}

The axioms for the class $\cal{M}$ in (M1)-(M5) can then be rewritten in equivalent form as properties of the classifier $K$, see~\cite{PittsAM:aximct} for details and~\cite{HylandM:firssd,RosoliniG:phd} for related ideas in synthetic domain theory.
Observe that the factorization of $\top \co 1 \to \Omega$ via $K$ gives us a pullback diagram
\begin{equation} \label{partial-map-classifier:0}
\begin{gathered}
\xymatrix{
  1
  \ar@{=}[r]
  \ar[d]
  \pullback{dr}
&
  1
  \ar[d]^{\top}
\\
  K
  \ar@{>->}[r]
&
  \Omega
\rlap{,}}
\end{gathered}
\end{equation}

Let us now fix $Y \in \cal{E}$ and work in the slice category $\cal{E}_{/Y}$.
Let $t \co Y \to K \times Y$ denote the classifier for $\cal{I}_{/Y}$.
Let $\pi \co K \times Y \to Y$ denote the map to the terminal object.
The \emph{$\cal{M}$-partial map classifier} $P_Y$ relative to $Y$ is defined by letting $P_Y \defeq \pi_! \cc t_*$~\cite[A2.4]{johnstone:elephant}.
Note that this is the polynomial endofunctor, in the sense of~\cite{gambino-kock}, associated to the diagram
\[
\xymatrix{
  Y
&
  Y
  \ar[l]_-{\id_K}
  \ar[r]^-{t}
&
  K \times Y
  \ar[r]^-{\pi}
&
  Y
}
\]
Observe that we have $t^* t_* = \Id$ since $t$ is monic, giving rise to a pullback square
\begin{equation} \label{partial-map-classifier:1}
\begin{gathered}
\xymatrix{
  X
  \ar[r]^-{\eta_X}
  \ar[d]
  \pullback{dr}
&
  P_Y X
  \ar[d]
\\
  Y
  \ar[r]_-{t}
&
  K \times Y
}
\end{gathered}
\end{equation}
for every $X \in \cal{E}_{/Y}$.

\begin{theorem} \label{thm:part-map-classifier}
Giving the structure of a uniform trivial fibration on a map $f \co X \to Y$ is equivalent to giving a diagonal filler for the diagram
\begin{equation*}
\begin{gathered}
\xymatrix{
  X
  \ar@{=}[r]
  \ar[d]_{\eta_X}
&
  X
  \ar[d]^{f}
\\
  P_Y X
  \ar[r]
  \ar@{.>}[ur]
&
  Y
\rlap{,}}
\end{gathered}
\end{equation*}
\ie a retraction of $\eta_X$ in $\cal{E}_{/Y}$.
\end{theorem}

\begin{proof}
Let $\cal{I}$ be the full subcategory of $\cal{M}$ spanned by maps with a representable presheaf as codomain and recall from the proof of \cref{rem-lift-suitable} that we have $\liftr{\cal{M}} = \liftr{\cal{I}}$.
Let $u \co \cal{I} \to \cal{E}^\to$ denote the inclusion.
Following the first step of Garner's small object argument, a lift of a map $f \co X \to Y$ to an element of $\liftr{\cal{I}}$ is given by a diagonal filler in the canonical square $(\Lan_u u)(f) \to f$.
Unfolding the left Kan extension, we have
\begin{align*}
(\Lan_u u)(f)
&=
\colim_{i \in \cal{I}, u_i \to f} u_i
\\&=
\colim_{\sigma : \yon C \to K \times Y,\ \cal{E}_{/Y}(t^* \sigma, X)} \sigma^* t
\\&=
\colim_{\sigma : \yon C \to P_Y(X)} \sigma^* \eta_X
\\&=
\eta_X
.
\end{align*}
Here, the penultimate step uses the pullback square~\eqref{partial-map-classifier:1} and the last step uses the Yoneda lemma and preservation of colimits by pullback.
\end{proof}

\begin{theorem} \label{unif-vs-non-unif} \leavevmode
\begin{enumerate}[(i)]
\item A map can be equipped with the structure of a uniform trivial fibration if and only if it is a trivial fibration, \ie it has the right lifting property with respect to the maps in $\cal{M}$.
\item A map can be equipped with the structure of a uniform fibration if and only if it is a fibration, \ie it has the right lifting property with respect to the generating trivial cofibrations.
\end{enumerate}
\end{theorem}

\begin{proof}
For part (i), since $1 \to K$ is in $\cal{M}$, as proved in \cref{partial-map-classifier}, then also $\eta$ has components in $\cal{M}$ by~\eqref{partial-map-classifier:1}.
Part (ii) follows from part (i) and \cref{prod-exp-general}.
\end{proof}

\begin{example} \label{elegant-reedy}
\cref{unif-vs-non-unif} applies to $\SSet$ and $\CSet$, showing that a map can be equipped with the structure of a (trivial) Kan fibration if and only if it is a (trivial) Kan fibration in the usual sense.
\end{example}

\begin{remark} \label{why-uniform}
Since \cref{unif-vs-non-unif} is proved constructively, it is natural to wonder whether it is possible to develop constructively our theory in the non-algebraic setting, at least for $\SSet$ or~$\CSet$.
The key obstacle to this is that, while the special case of the algebraic result in \cref{rem-lift-suitable} admits a constructive proof in $\SSet$ and $\CSet$, as we explained in~\cref{rem:constructive-small-object}, its non-algebraic counterpart does not seem to provable be constructively, even in the case when $\cal{M}$ consists of all monomorphisms (\cf \cref{thm:generation-presheaf-cisinski,rem:colimit-decomp}).
And such a result is essential for the development since it provides the suitable wfs on which the definition of a fibrations is based.
\end{remark}

\bibliographystyle{plain}
\bibliography{}

\end{document}